\documentclass[11pt]{article}
\usepackage{epsfig}
\usepackage{amssymb,amsmath,amsthm,amscd}
\usepackage{latexsym}
\usepackage{fixltx2e}
\usepackage{mathrsfs}
\usepackage{verbatim}

\newtheorem{thm}{Theorem}[section]
\newtheorem{prop}[thm]{Proposition}

\newtheorem{lem}[thm]{Lemma}

\theoremstyle{definition}
\newtheorem{defn}[thm]{Definition}

\newcounter{labelflag} 

\newcommand{\Label}[1]{
                       \ifnum\thelabelflag=1
                          \ifmmode
                             \makebox[0in][l]{\qquad\fbox{\rm#1}}
                          \else
                             \marginpar{\vspace{0.7\baselineskip}
                                        \hspace{-1.1\textwidth}
                                        \fbox{\rm#1}}
                          \fi
                       \fi
                       \label{#1} }

\newcommand{\be}{\begin{equation}}
\newcommand{\ee}{\end{equation}}

\newcommand{\abs}[1]{|#1|}
\newcommand{\norm}[1]{\|#1\|}

\newcommand{\R}{\mathbb{R}}
\newcommand{\N}{\mathbb{N}}

\def \calf {{  {\mathcal{F}} }}

\def \cala {{  {\mathcal{A}}  }}
\def \calb {{  {\mathcal{B}}  }}
\def \cald {{  {\mathcal{D}} }}

 \def \o {{  \mathcal{O}  }}

 \def  \ltwo { {L^2 ({\R^n}) } }
  \def  \lq { { L^q  ({\R^n}) } }
   \def \ii {\int_{{\R^n}}}

  \def \eps{ { \varepsilon }}

\begin{document}

\begin{titlepage}
\title{\Large\bf   
  Pullback Attractors of Non-autonomous Stochastic
  Degenerate Parabolic Equations on Unbounded Domains }
\vspace{7mm}

\author{
 Andrew Krause \quad  and \quad 
 Bixiang Wang  
\vspace{1mm}\\
Department of Mathematics, New Mexico Institute of Mining and
Technology \vspace{1mm}\\ Socorro,  NM~87801, USA \vspace{1mm}\\
Email: akrause@nmt.edu,  \quad \quad    bwang@nmt.edu}

\date{}
\end{titlepage}

\maketitle

\medskip

\medskip

\begin{abstract}
This paper is concerned with
pullback attractors of  
  the   stochastic
  $p$-Laplace  equation 
  defined on the entire space $\R^n$.
  We first establish the asymptotic compactness
  of the   equation in $\ltwo$ and  then prove
  the existence and uniqueness of  non-autonomous
  random attractors.
  This attractor is pathwise periodic if 
  the non-autonomous deterministic
  forcing is time periodic.
  The  difficulty of non-compactness of Sobolev embeddings on $\R^n$
  is overcome by the uniform smallness of  solutions outside a bounded domain.
\end{abstract}



\baselineskip=1.35\baselineskip

\section{Introduction}
\setcounter{equation}{0}

In this paper,  we study  random attractors
of  the non-autonomous stochastic 
$p$-Laplace equation defined on $\R^n$.
Suppose 
$(\Omega, \calf, P,  \{\theta_t\}_{t \in \R})$
 is a metric dynamical system where 
 $(\Omega, \calf, P)$ is a probability space
 and  $\{\theta_t\}_{t \in \R}$ is a measure-preserving
 transformation group on $\Omega$.
  Given $\tau \in\R$ and $\omega \in \Omega$, 
consider the   stochastic     equation 
defined for   $ x \in {\R^n}$ and   $t > \tau$, 
\be
  \label{intr1}
  {\frac {\partial u}{\partial t}}  + \lambda u
  - {\rm div} \left (|\nabla u |^{p-2} \nabla u \right )
  =f(t,x,u ) 
  +g(t,x) 
  +\alpha \eta (\theta_t \omega) u   +\eps  h(x) {\frac {dW}{dt}},
  \ee 
  where $p\ge 2$,  $ \alpha>0  $,  $\lambda>0$,      $\eps>0 $,
  $f$ is a time dependent  nonlinearity, $g$  and $h$  are
  given functions,  
$\eta$ is a   random variable   and
$W$  is a    Wiener process on  $(\Omega, \calf, P)$.
The $p$-Laplace equation has been used
to model a variety of physical phenomena.
For instance,  in fluid dynamics,
the motion of    non-Newtonian fluid
is governed by the $p$-Laplace equation with $p \neq 2$.
This equation  is  also used to study  flow in  porous media 
and nonlinear elasticity (see \cite{lion1} for more details).

   In this paper, we will investigate the asymptotic behavior
   of solutions of the $p$-Laplace equation \eqref{intr1}
   driven by deterministic as well as stochastic
   forcing.
   If $f$  and $g$ do not depend on time,  then we
   call   \eqref{intr1}  an autonomous stochastic equation.
   In the autonomous  case,   the existence of   random attractors
   of \eqref{intr1}    has been  established recently
   in   \cite{gess1, gess2, gess3}
   by variational methods  under the condition that
    the growth rate
   of the nonlinearity $f$ is  not  bigger than $p$.
   This result   has been extended in \cite{wan8}
    to the case where    $f$ is non-autonomous and has a
    polynomial growth of any order. 
    Note that in all  papers mentioned above, 
    the  $p$-Laplace equation is defined in a bounded domain
    where compactness of Sobolev embeddings is available.
    As far as we are aware, there is no  existence result on random
    attractors for   the stochastic $p$-Laplace equation defined on unbounded
    domains.  The goal of this paper is to
    overcome the non-compactness of Sobolev embeddings on $\R^n$
    and prove the existence and uniqueness of random
    attractors for \eqref{intr1} in $\ltwo$. 
    More precisely,   we will 
    show by a cut-off technique   that the tails of  solutions of \eqref{intr1} are uniformly
    small outside a bounded domain for large times.
    We then use this fact and 
      the compactness of solutions in bounded domains
      to establish the asymptotic compactness of solutions
      in $\ltwo$.
      By the asymptotic compactness and   absorbing sets of the equation,
      we  can obtain the existence and uniqueness  of random attractors.
      This  random  attractor    is   pathwise periodic if $f(t,x,u)$ and $g(t,x)$ are 
      periodic in $t$.

 It is worth  mentioning   that
 the definition of random attractor
 for autonomous stochastic systems
 was initially introduced in \cite{cra2, fla1,  schm1}.
 Since then,   
 such  attractors for autonomous stochastic PDEs
 have  been studied in 
  \cite{beyn1,  car1, car2, 
   car3, car4, car5, chu2, chu3, cra1, cra2, 
    fla1, gar1, gar2, gar3,  gess1, gess2, gess3,
      huang1, kloe1, schm1, shen1}  
       in bounded domains,
       and in     \cite{bat2, 
    wan1,  wan2}
         in  unbounded domains.
       For non-autonomous
       stochastic PDEs, the reader is referred to
         \cite{adi1, bat3,  car6, dua3, gess2, gess3,  kloe2,  wan5,
      wan7,  wan8}
     for   existence of  random   attractors.
      
      The   rest of this article  consists of three sections.
    In the next section,  we  prove  the well-posedness
      of  \eqref{intr1} in $\ltwo$ under certain conditions.   
       Section 3 is devoted to uniform estimates of solutions
       for large times which include  the  estimates on the tails of solutions
       outside a bounded domain. 
       In the last section,  we derive the existence and  uniqueness
       of    random attractors for 
       the non-autonomous stochastic equation \eqref{intr1}.

        The following notation will be used throughout this paper:
$\| \cdot \|$  for the norm  of $\ltwo$ and 
  $(\cdot, \cdot)$      for  its 
   inner product.
The norm of $L^p(\R^n)$ is usually  written  as $\| \cdot \|_p$
and the 
  norm  of a   Banach space  $X$
 is written as    $\|\cdot\|_{X}$.
 The symbol   $c$ or  $c_i$ ($i=1, 2, \ldots$) is used
 for a general positive 
      number  which may change from line to line.
 
 Finally, we recall the following inequality which will be used
 frequently in the sequel: 
\begin{equation}
\label{p_inequality}
\norm{u}^p_p \leq \frac{q-p}{q-2}\norm{u}^2 + \frac{p-2}{q-2}\norm{u}^q_q,
\ee
 where   $2 < p < q$  and  $u \in 
\ltwo \bigcap \lq$.

\section{Cocycles  Associated with Degenerate Equations} 
\setcounter{equation}{0}

In this section, we  first establish the well-posedness of equation
\eqref{intr1} in $\ltwo$, and then define a continuous cocycle
for the   stochastic equation. This step is necessary for us
to investigate the asymptotic behavior of solutions.
 
Let   $(\Omega, \calf, P)$ be the standard probability space
where 
$ 
\Omega = \{ \omega   \in C(\R, \R ):  \omega(0) =  0 \}
$,   $\calf$   is 
 the Borel $\sigma$-algebra induced by the
compact-open topology of $\Omega$  and $P$
is   the   Wiener
measure on $(\Omega, \calf)$.   
Denote by  $\{\theta_t\}_{t \in \R}$ 
 the  family of shift operators given by 
$$
 \theta_t \omega (\cdot) = \omega (\cdot +t) - \omega (t)
\quad  \mbox{ for  all  } \   \omega \in \Omega \ 
\mbox{ and } \  t \in \R.
$$
From \cite{arn1} we know that 
 $(\Omega, \calf, P,  \{\theta_t\}_{t \in \R})$
 is a metric dynamical system.
    Given $\tau \in\R$ and $\omega \in \Omega$, 
consider the  following  stochastic     equation 
defined for   $ x \in {\R^n}$ and   $t > \tau$, 
\be
  \label{seq1}
  {\frac {\partial u}{\partial t}}  + \lambda u
  - {\rm div} \left (|\nabla u |^{p-2} \nabla u \right )
  =f(t,x,u ) 
  +g(t,x) 
  +\alpha \eta (\theta_t \omega) u   +\eps  h(x) {\frac {dW}{dt}}
  \ee 
  with  initial condition
 \be\label{seq2}
 u( \tau, x ) = u_\tau (x),   \quad x\in  {\R^n},
 \ee
 where $p\ge 2$,  $ \alpha>0  $,  $\lambda>0$,      $\eps>0 $,
$g \in L^2_{loc}(\R, \ltwo)$,  $h \in H^2({\R^n})$, 
$\eta$ is an integrable tempered  random variable   and
$W$  is a  two-sided real-valued Wiener process on  $(\Omega, \calf, P)$.
We  assume the nonlinearity 
$f: \R \times {\R^n} \times \R$ 
$\to \R$ is continuous and satisfies, 
for all
$t, s \in \R$   and  $x \in {\R^n}$, 
\be 
\label{f1}
f (t, x, s) s \le - \gamma
 |s|^q + \psi_1(t, x),  
\ee
\be 
\label{f2}
|f(t, x, s) |   \le \psi_2 (t,x)  |s|^{q-1} + \psi_3 (t, x),
\ee
\be 
\label{f3}
{\frac {\partial f}{\partial s}} (t, x, s)   \le \psi_4 (t,x),
\ee
where $\gamma>0$  and $ q \ge p $
are constants,
$\psi_1 \in L^1_{loc} (\R,  L^1( {\R^n}) )$,
$\psi_2, \psi_4 \in L^\infty_{loc} (\R, L^\infty ( {\R^n}))$,
and $\psi_3 \in L^{q_1}_{loc} (\R, L^{q_1} (  {\R^n}))$.
From now on,   we  always  
assume  $h \in H^2({\R^n}) \bigcap  W^{2,q} ({\R^n})$
and   use $p_1$ and $q_1$
to denote 
  the conjugate exponents
of $p$ and $q$, respectively. 
Since  $h \in H^2({\R^n}) \bigcap  W^{2,q} ({\R^n})$
and $q \ge p$, by \eqref{p_inequality} we find
$h \in   W^{2,p} ({\R^n})$.

To define a random dynamical system for
  \eqref{seq1},  we need to transfer the stochastic equation
  to a pathwise  deterministic  system.
  As usual,  let $z$ be the 
  random variable  given by:
$$
z ( \omega)=   - \lambda  \int^0_{-\infty} e^{\lambda  \tau}    \omega  (\tau) d \tau,
\quad \omega \in \Omega.
$$
It  follows from \cite{arn1} that
 there exists a $\theta_t$-invariant 
 set $\widetilde{\Omega}  $  
 of full
 measure 
 such that  
  $z(\theta_t \omega)$  is
 continuous in $t$ 
and $ \lim\limits_{t \to \pm \infty}   {\frac { |z(\theta_t \omega)|}{|t|}}
= 0 $
for all  $\omega \in \widetilde{\Omega}$.
We also  assume    $\eta(\theta_t \omega)$  is
pathwise 
 continuous  for each fixed $\omega \in  \widetilde{\Omega}$.
For convenience,  we will  denote  
$\widetilde{\Omega}$    by   $\Omega$
in the sequel. 
  Let $u(t, \tau, \omega, u_\tau)$
   be a solution of problem \eqref{seq1}-\eqref{seq2}
   with initial condition
   $u_\tau$ at initial time $\tau$,  and  define
 \be
 \label{uv}
 v(t, \tau, \omega, v_\tau)
   =   u(t, \tau, \omega, u_\tau)
    - \eps  h(x) z(\theta_{t} \omega)
 \quad \mbox{with }   \
 v_\tau =   u_\tau - \eps  h  z(\theta_\tau \omega).
 \ee
 By \eqref{seq1} and \eqref{uv}, 
 after simple calculations,
  we get
 \be 
 \label{veq1}
\frac{\partial v}{\partial t}-\text{div}
\left (
\abs{\nabla
 (v+\eps h(x)z(\theta_{t} \omega))}^{p-2}
 \nabla (v+\eps h(x)z(\theta_{t} \omega))
 \right )
 + \lambda v 
 $$
 $$
 = f(t,x,v+\eps h(x)z(\theta_{t} \omega))
+ g(t,x) + \alpha\eta(\theta_t\omega)v
+ \alpha \eps  \eta(\theta_t\omega)     z(\theta_t\omega)  h,
\ee
with initial   condition
\be 
 \label{veq2}
v(\tau, x )=v_{\tau}(x ), \quad  x\in {\R^n}.
\ee

In what follows,  we first prove the well-posedness
of problem  \eqref{veq1}-\eqref{veq2}
in $\ltwo$, and then define
a cocycle   for \eqref{seq1}-\eqref{seq2}.

 \begin{defn}
 \label{defv}
 Given $\tau \in \R$, $\omega \in \Omega$
 and
 $v_\tau \in \ltwo$, 
 let  $v(\cdot, \tau, \omega, v_\tau)$: $[\tau, \infty) \to \ltwo$
 be a continuous function with
  $v
 \in L^p_{loc}([\tau, \infty), W^{1,p}({\R^n}) ) \bigcap
 L^q_{loc} ([\tau, \infty), \lq)
 $  and 
 ${\frac {dv}{dt}}
 \in L^{p_1}_{loc}([\tau, \infty),  (W^{1,p}) ^*) 
 + 
 L^{2}_{loc} ( [\tau, \infty),  \ltwo )
 + 
 L^{q_1}_{loc} ( [\tau, \infty), L^{q_1} ({\R^n} ) )$.
 We say $v$ is a solution of   \eqref{veq1}-\eqref{veq2}
 if  $ v(\tau, \tau, \omega, v_\tau) =v_\tau$  
   and     for every
   $\xi \in   W^{1,p}({\R^n}) \bigcap \ltwo  \bigcap  \lq $,
 $$
 {\frac {d}{dt}}   (v, \xi)
 +\int_{\R^n} 
 \abs{\nabla
 (v+\eps h z(\theta_{t} \omega))}^{p-2}
 \nabla (v+\eps h z(\theta_{t} \omega) )
\cdot   \nabla \xi dx 
+ \left (\lambda - \alpha\eta(\theta_t\omega) \right )
 (v, \xi)
 $$
 $$
 =
 \int_{\R^n} f(t,x,v+\eps  h z(\theta_{t} \omega)) \xi dx
+ (g(t, \cdot), \xi) 
+ \alpha \eps  \eta(\theta_t\omega)     z(\theta_t\omega)  
(h, \xi)
   $$
 in the sense of distribution on $[\tau, \infty)$.
 \end{defn}
 
 Next, we prove the existence and uniqueness of solutions
 of \eqref{veq1}-\eqref{veq2} in $\ltwo$. To this end,
 we  set $\o_k = \{ x \in \R^n:  |x| < k \}$ for each $k \in \N$
 and consider the following equation defined in $\o_k$:
 \be 
 \label{veqk1}
\frac{\partial v_k}{\partial t}-\text{div}
\left (
\abs{\nabla
 (v_k+\eps h(x)z(\theta_{t} \omega))}^{p-2}
 \nabla (v_k+\eps h(x)z(\theta_{t} \omega))
 \right )
 + \lambda v_k 
 $$
 $$
 = f(t,x,v_k+\eps h(x)z(\theta_{t} \omega))
+ g(t,x) + \alpha\eta(\theta_t\omega)v_k
+ \alpha \eps  \eta(\theta_t\omega)     z(\theta_t\omega)  h,
\ee
with boundary condition
\be\label{veqk2}
v_k (t, x) = 0 \quad   \text{for all } t>\tau \text{ and }  |x| =k
\ee
and  initial   condition
\be 
 \label{veqk3}
v(\tau, x )=v_{\tau}(x )  \quad  \text{ for all }  x\in \o_k.
\ee
 By the arguments in \cite{wan8},  one can show that
 if  \eqref{f1}-\eqref{f3} are fulfilled,  then
 for every $\tau \in \R$
 and   $\omega \in \Omega$,   system  \eqref{veqk1}-\eqref{veqk3}
  has  a unique  solution
 $v_k(\cdot, \tau, \omega, v_\tau)  $
 in the sense of Definition \ref{defv} with $\R^n$  replaced by
 $\o_k$.   Moreover, 
 $v_k(t, \tau, \omega, v_\tau)$  is  $(\calf, \calb (L^2(\o_k)))$-measurable
 with respect to $\omega \in \Omega$. 
 We now investigate 
  the limiting behavior of $v_k$
 as $k \to \infty$.
  For convenience,  we   write  $V_k
  =W_0^{1,p} (\o_k)$    and $V= W^{1,p} (\R^n)$.
  Let   
 $A$: $ V_k  \to V_k^*  $ be  the operator  given by
\be\label{Av}
 (A(v_1),  v_2 )_{(V_k^*, V_k)}
 = \int_{\o_k} | \nabla v_1|^{p-2} \nabla v_1 \cdot \nabla v_2 dx,
 \quad  \mbox{for all} \ v_1, v_2 \in V_k,
\ee
 where $ (\cdot,   \cdot )_{(V_k^*, V_k)}$ is the duality pairing 
 of $V_k^*$  and $V_k$.
 Note that $A$  is a monotone operator
 as in \cite{show1} and  $A: V \to V^*$  is also well defined by
 replacing $\o_k$ by $\R^n$ in \eqref{Av}.
 The following uniform estimates on $v_k$  are useful.
  
 \begin{lem}
 \label{lemvk}
 Suppose \eqref{f1}-\eqref{f3} hold. Then for every  $T>0$, 
 $\tau \in \R$, $\omega \in \Omega$  and $v_\tau \in \ltwo$,
 the solution $v_k(t, \tau, \omega, v_\tau)$ of system
 \eqref{veqk1}-\eqref{veqk3} has the properties:
$$
\{v_k\}_{k=1}^\infty
\ \mbox{ is bounded in } \
 L^\infty (\tau, \tau +T;  L^2(\o_k) ) \bigcap L^q(\tau, \tau +T;  L^q(\o_k) )
 \bigcap L^p(\tau, \tau +T; V_k),
$$
$$
 \{A( v_k+\eps h z(\theta_{t} \omega) )\}_{k=1}^\infty
\ \mbox{ is bounded in } \
  L^{p_1}(\tau, \tau +T; V_k^*)  \  \text{ with }  \ {\frac 1{p_1}} + {\frac 1p} =1,
  $$
  $$
  \{   f(t,x,v_k+\eps h z(\theta_{t} \omega))  \}_{k=1}^\infty
 \  \mbox{ is bounded in }  \
 L^{q_1} (\tau, \tau +T; L^{q_1} ({\o_k})),
 \quad  {\frac 1{q_1}} + {\frac 1q} =1,
$$
and
$$
  \left \{ {\frac {dv_k}{dt}}  \right \} 
  \mbox{ is bounded in }   
  L^{p_1}(\tau, \tau +T; V_k^*)   +  
    L^{2} (\tau, \tau +T; L^{2} ({\o_k}))
     +  
 L^{q_1} (\tau, \tau +T; L^{q_1} ({\o_k})).
 $$ 
 \end{lem}
  
  \begin{proof}
  By \eqref{veqk1}  we get
 \be
 \label{plvk1_1}
\frac{1}{2}\frac{d}{dt}\norm{v_k}^2 
+  \int_{\o_k} \abs{\nabla (v_k
+\eps h z(\theta_{t} \omega))}^{p-2}
\nabla (v_k+\eps h z(\theta_{t} \omega)  )
 \cdot  \nabla v_k dx
+ \lambda \norm{v_k}^2
$$
$$
= \int_{\o_k}
 f(t,x,v_k+\eps h z(\theta_{t} \omega) ) v_k dx
 + (g(t), v_k)
 + \alpha\eta(\theta_t\omega) \norm{v_k}^2
 +  \alpha\eps \eta(\theta_t\omega)  z(\theta_t\omega ) (h, v_k)   .
\ee
For the second  term on the left-hand side of \eqref{plvk1_1},  
by Young\rq{}s inequality we  obtain
$$ \int_{\o_k} \abs{\nabla (v_k
+\eps h z(\theta_{t} \omega))}^{p-2}
\nabla (v_k+\eps h z(\theta_{t} \omega)  ) \cdot  \nabla v_k dx
$$
$$
  = \int_{\o_k}\abs{\nabla (v_k +\eps h z(\theta_{t} \omega))}^{p}dx 
  - \int_{\o_k} \abs{\nabla (v_k+\eps h z(\theta_{t} \omega))}^{p-2}
  \nabla(v_k +\eps h z(\theta_{t}\omega) ) 
  \cdot \nabla  ( \eps h z(\theta_{t} \omega ) ) dx 
  $$
  \be\label{plvk1_2}
  \ge
  {\frac 12}  \int_{\o_k}\abs{\nabla (v_k +\eps h z(\theta_{t} \omega))}^{p}dx 
   -  c_1 \abs{\eps z(\theta_{t} \omega) }^p \| \nabla h \|^p_p.
  \ee
  For the first   term on the right-hand side of \eqref{plvk1_1},
 by  \eqref{f1} and \eqref{f2} we get
   \be
    \label{plvk1_3}
\int_{\o_k}
f(t,x,v_k+\eps h  z(\theta_{t} \omega)) v_k dx 
$$
$$
=  \int_{\o_k}f(t,x,v_k+\eps h z(\theta_{t} \omega))
(v_k +\eps h z(\theta_{t} \omega))dx 
- \eps z(\theta_{t} \omega) \int_{\o_k}
f(t,x,v_k +\eps h z(\theta_{t} \omega))h(x)dx
$$
$$
\leq -\gamma\int_{\o_k}
\abs{v_k+\eps h z(\theta_{t} \omega))}^q dx 
+ \int_{\o_k}\psi_1(t,x)dx
$$
$$
+
\int_{\o_k}\psi_2(t,x)
\abs{v_k+\eps h z(\theta_{t} \omega))}^{q-1}
\abs{\eps h z(\theta_{t} \omega)}dx 
+ \int_{\o_k}\psi_3(t,x) \abs{\eps h z(\theta_{t} \omega)}dx
$$
$$
\leq -\frac{\gamma}{2}\norm{v_k+\eps h z(\theta_{t} \omega))}^q_q 
+ \norm{\psi_1(t)}_{1}
+\norm{\psi_3(t)}_{q_1}^{q_1} 
+ c_2 \int_{\o_k}\abs{\eps h z(\theta_{t} \omega)}^q dx.
\ee
 By Young\rq{}s inequality   we obtain 
\be \label{plvk1_4}
\int_{\o_k} g(t,x) v_k dx  + \alpha \eps \eta(\theta_t\omega)z(\theta_t\omega)
\int_{\o_k} h(x) v_k dx  
$$
$$
\leq  
  \frac{4}{\lambda} \abs{ \alpha \eps \eta(\theta_t\omega)
z(\theta_t\omega)}^2 \norm{h} ^2
 + \frac{4}{\lambda}  \norm{g (t)}^2  + \frac{\lambda}{8}\norm{v_k}^2.
\ee
It follows   from \eqref{plvk1_1}-\eqref{plvk1_4}  that
  \be
  \label{plvk1_6}
\frac{d}{dt}\norm{v_k}^2 
+ {\frac{7}{4}}  \lambda \norm{v_k}^2 
+ \int_{\o_k} \abs{\nabla (v_k+\eps hz(\theta_{t} \omega))}^p_p  dx
+ \gamma
\int_{\o_k} \abs{v_k+\eps h z(\theta_{t} \omega))}^q_q dx
$$
$$
\leq 2\alpha\eta(\theta_t\omega)\norm{v_k}^2 
+ c_3 \left (  \abs{\eps z(\theta_{t} \omega)}^p 
+   \abs{\eps z(\theta_{t} \omega)}^q 
+     \abs{\alpha \eps \eta(\theta_t\omega)z(\theta_t\omega) }^2
\right )
$$
$$
 + c_4 \left (
 \norm{g(t)}^2\
 + \norm{\psi_1(t)}_{1}+   \norm{\psi_3(t)}_{q_1}^{q_1}
 \right ) .
\ee
 Multiplying \eqref{plvk1_6}
  by $e^{\frac{7}{4}\lambda t-2\alpha\int^{t}_{0}
  \eta(\theta_{r}\omega)dr}$, 
  and then  integrating from $\tau $ to $t$, we get 
  \be 
  \label{plvk1_8}
\norm{v_k(t, \tau,\omega,v_{\tau})}^2 
+ \int_{\tau}^{t}e^{\frac{7}{4}\lambda(s-t)
-2\alpha\int_{t}^{s}\eta(\theta_{r}\omega)dr}
\int_{\o_k} \abs{\nabla (v_k(s, \tau,\omega,v_{\tau})
+\eps h  z(\theta_{s} \omega))}^p_p dxds
$$
$$
+ \gamma\int_{\tau}^{t}e^{\frac{7}{4}\lambda(s-t)
-2\alpha\int_{t}^{s}\eta(\theta_{r}\omega)d r}
\int_{\o_k}\abs{v_k(s, \tau,\omega,v_{\tau})
+\eps h  z(\theta_{s} \omega))}^q_q dxds
$$
$$
\le  c_3  \int_{\tau}^{t}e^{\frac{7}{4}\lambda(s-t)-2\alpha\int_{t}^{s}
\eta(\theta_{r}\omega)d r}
\left (
\abs{\eps z(\theta_{s} \omega)}^p 
+  \abs{\eps z(\theta_{s} \omega)}^q 
+  \abs{\alpha \eps \eta(\theta_{s}\omega)z(\theta_s\omega)}^2
\right ) ds
$$
$$
+ c_4 \int_{\tau}^{t}e^{\frac{7}{4}
\lambda(s-t)-2\alpha\int_{t}^{s}\eta(\theta_{r}\omega)d r}
\left (
\norm{g(s)}^2
 + \norm{\psi_1(s)}_{1}+\norm{\psi_3(s)}_{q_1}^{q_1} 
\right )ds 
$$
$$
+ e^{\frac{7}{4}\lambda(\tau-t)-2\alpha\int_{t}^{\tau}
\eta(\theta_{r}\omega)d r}\norm{v_{\tau}}^2 .
\ee
By \eqref{plvk1_8} we get
\be\label{plvk1_10}
\{v_k\} 
\ \mbox{ is bounded in } \
 L^\infty (\tau, \tau +T;  L^2(\o_k) ) \bigcap L^q(\tau, \tau +T;  L^q(\o_k) )
 \bigcap L^p(\tau, \tau +T; V_k).
\ee
By \eqref{f2} and \eqref{plvk1_10} we obtain
\be\label{plvk1_11}
  \{   f(t,x,v_k+\eps h z(\theta_{t} \omega))  \}_{k=1}^\infty
 \  \mbox{ is bounded in }  \
 L^{q_1} (\tau, \tau +T; L^{q_1} ({\o_k}) ).
\ee
By \eqref{Av}  and \eqref{plvk1_10} we get
\be\label{plvk1_12}
 \{A( v_k+\eps h z(\theta_{t} \omega)   ) \}_{k=1}^\infty
\ \mbox{ is bounded in } \
  L^{p_1}(\tau, \tau +T; V_k^*) .
  \ee
   By \eqref{plvk1_10}-\eqref{plvk1_12}  it follows from 
   \eqref{veqk1} that 
$$
  \left \{ {\frac {dv_k}{dt}}  \right \} 
  \mbox{ is bounded in }   
  L^{p_1}(\tau, \tau +T; V_k^*)   +  
    L^{2} (\tau, \tau +T; L^{2} ({\o_k}))
     +  
 L^{q_1} (\tau, \tau +T; L^{q_1} ({\o_k})),
$$
which completes   the proof.
  \end{proof}

 The next lemma is concerned with the well-posedness
 of \eqref{veq1}-\eqref{veq2} in $\ltwo$.

 \begin{lem}
 \label{exiv}
 Suppose \eqref{f1}-\eqref{f3} hold. Then for every 
 $\tau \in \R$, $\omega \in \Omega$  and $v_\tau \in \ltwo$,
 problem \eqref{veq1}-\eqref{veq2} has a unique solution
 $v(t, \tau, \omega, v_\tau)$ in the sense of Definition \ref{defv}.
 In addition,  $v(t, \tau, \omega, v_\tau)$
 is  ($\calf, \calb (\ltwo))$-measurable
   in $\omega$ and continuous in
  $v_\tau$ in $\ltwo$   and   satisfies 
 $$
 {\frac d{dt}} \norm{ v(t, \tau, \omega, v_\tau) }^2
 + 2 \left (
  \lambda - \alpha \eta  (\theta_{t} \omega)
 \right ) \norm{ v }^2
 +2   \norm{ \nabla
 (v+\eps h z(\theta_{t} \omega))}^p_p
 $$
 $$
 = 2 \eps  z(\theta_{t} \omega)  \ii 
\abs{\nabla
 (v+\eps h z(\theta_{t} \omega))}^{p-2}
 \nabla (v+\eps h z(\theta_{t} \omega))
\cdot \nabla h dx
$$
\be\label{enlem1}
+2  \ii  f(t,x,v+\eps h z(\theta_{t} \omega)) v dx
+ 2 ( g(t), v)   
+ 2 \alpha \eps  \eta(\theta_t\omega)     z(\theta_t\omega)  (h, v)
\ee
  for almost  all $t \ge \tau$.
  \end{lem}
  
  \begin{proof} 
  Let $T>0$,  $t_0 \in [\tau, \tau +T]$ and $v_k
  (t, \tau, \omega, v_\tau)$ be the solution of 
  system \eqref{veqk1}-\eqref{veqk3} defined
  in $\o_k$.   Extend $v_k$
  to the entire space $\R^n$
  by setting $v_k =0$    on $\R^n \setminus \o_k$
  and denote this  extension still by
  $v_k$.
    By Lemma \ref{lemvk} we find that 
 there exist
 $\widetilde{v}\in \ltwo$,
 $v \in L^\infty (\tau, \tau +T; \ltwo)  
 \bigcap L^p(\tau, \tau +T; V) \bigcap L^q(\tau, \tau +T; \lq) $,
 $\chi_1
 \in  L^{q_1} (\tau, \tau +T;  L^{q_1}({\R^n}) ) $,
  $\chi_2
 \in  L^{p_1} (\tau, \tau +T;  V^* ) $
 such that, up to a subsequence,
 \be\label{pexiv_1}
 v_k \to v \ \mbox{weak-star  in } \ 
 L^\infty (\tau, \tau +T; \ltwo ),
 \ee
  \be\label{pexiv_2}
 v_k \to v \ \mbox{weakly  in } \ 
 L^p (\tau, \tau +T;  V )
 \ \mbox{ and } \  L^q(\tau, \tau +T; \lq),
 \ee
  \be\label{pexiv_3}
 A(  v_k+\eps h z(\theta_{t} \omega))   \to  \chi_2  \ \mbox{weakly  in } \ 
 L^{p_1} (\tau, \tau +T;  V^* ),
 \ee 
  \be\label{pexiv_4}
  f(t,x,v_k +\eps h z(\theta_{t} \omega))   \to  \chi_1  \ \mbox{weakly  in } \ 
 L^{q_1} (\tau, \tau +T; L^{q_1} ({\R^n}) ),
 \ee
 and
 \be\label{pexiv_5}
 v_k(t_0,\tau, \omega, v_\tau)
 \to  {\widetilde{v}}
  \ \mbox{weakly  in } \ 
  \ \ltwo.
 \ee
 On the other hand, by 
 the  compactness
 of embedding $W^{1,p}(\o_k)
 \hookrightarrow L^2(\o_k)$ 
 and Lemma \ref{lemvk}, we can choose a further
 subsequence  (not relabeled)
  by a diagonal process such that
  for each $k_0 \in \N$,
   \be\label{pexiv_6}
 v_k \to v
 \ \mbox{ strongly  in } \ 
 L^2(\tau, \tau +T; L^2(\o_{k_0}) ).
 \ee
 By \eqref{veqk1}
 and \eqref{pexiv_1}-\eqref{pexiv_4}
 one can  show  that 
 for every $\xi \in V\bigcap \ltwo  \bigcap \lq$,
 $$
 {\frac {d}{dt}} (v, \xi)
 +    (\chi_2, \xi)_{(V^*, V)}
 +  (\lambda -\alpha \eta(\theta_t \omega) ) (v, \xi)
 $$
 \be\label{pexiv_8}
 =
   (\chi_1, \xi)_{(L^{q_1}, L^q)}
   +  (g(t), \xi) +
   \alpha \eps \eta(\theta_t \omega) z(\theta_t \omega)
   (h, \xi) 
\ee
 in the sense of distribution.
By \eqref{pexiv_8} we find 
\be\label{pexiv_9}
 {\frac {dv}{dt}} 
 =
-    \chi_2
 +    \chi_1 
 - (\lambda -\alpha \eta(\theta_t \omega) )v
 +g 
 +\alpha \eps \eta(\theta_t \omega) z(\theta_t \omega) h
\ee
in 
$L^{p_1}(\tau, \tau +T; V^*)
+  L^{q_1} (\tau, \tau +T; L^{q_1} ({\R^n}) )
+  L^{2} (\tau, \tau +T;  \ltwo  )$,
which along with the fact  
  $v \in  L^\infty(\tau, \tau +T;  \ltwo)
  \bigcap L^p(\tau, \tau +T; V)
\bigcap L^q(\tau, \tau +T; \lq)$
implies (see, e.g., 
 \cite{lion1})  that 
$v \in C ([\tau, \tau +T], \ltwo)$ and 
\be\label{pexiv_12}
{\frac 12} {\frac d{dt}} \| v \|^2
= ( {\frac {dv}{dt}},  v)_{( V^*+ L^{q_1} +L^2, 
V\bigcap L^q \bigcap L^2 )}
 \quad \mbox{for  almost all } \  
t \in (\tau, \tau +T).
\ee
By \eqref{pexiv_1}-\eqref{pexiv_6}, we can argue as
in \cite{wan8} to show that
\be\label{pexiv_20}
 \chi_2 = 
 A(  v+\eps h z(\theta_{t} \omega))  ,
 \quad
   \chi_1  = 
  f(t,x,v  +\eps h z(\theta_{t} \omega) ),
  \quad  v(\tau) =v_\tau \ 
  \text{ and } \   v(t_0) =  { \widetilde{v}}.
 \ee
 By   \eqref{pexiv_8} and \eqref{pexiv_20} 
 we  find that 
  $v$ is a solution of problem \eqref{veq1}-\eqref{veq2}
  in the sense of Definition \ref{defv}.
  On the other hand, 
  by   \eqref{pexiv_9} and \eqref{pexiv_20} 
  we see  that $v$ satisfies  energy equation
  \eqref{enlem1}.
 
 We next  prove   the uniqueness of solutions.
 Let  $v_1$ and $v_2$   be the solutions of \eqref{veq1}
 and    $ {\widetilde{v}} = v_1 -v_2$. Then we have
 $$
 {\frac {d{\widetilde{v}}}{dt}}
 +  
  A(v_1 + \eps h    z(\theta_t \omega) ) -A(v_2 +\eps h    z(\theta_t \omega) )  
 +\lambda {\widetilde{v}}
 $$
 $$
 =\alpha \eta (\theta_t \omega) {\widetilde{v}}
 +   f (t,x,  v_1 + \eps h    z(\theta_t \omega)  )
 -f (t,x,  v_2 + \eps h    z(\theta_t \omega)  ),
 $$
 which along with \eqref{f3} and the   monotonicity 
 of $A$ yields,   for all $t \in [\tau, \tau +T]$,
 $$ 
 {\frac {d}{dt}}
 \norm{\widetilde{v}}  ^2
 \le 2 \alpha  \eta (\theta_t \omega) \norm {\widetilde{v}}  ^2
 +2  
 \int_{\R^n} \psi_4 (t,x) | {\widetilde{v}} |^2 dx
  \le c   \| {\widetilde{v}} \|^2
 $$
 for some   positive constant $c$ depending on $\tau,
 T$ and $\omega$.
 By Gronwall\rq{}s lemma we get,
 for all $t \in [\tau, \tau +T]$,  
 \be\label{pexiv_30}
 \norm{ v_1(t, \tau, \omega, v_{1,\tau} ) - v_2(t, \tau, \omega, v_{2,\tau}) }^2
 \le e^{c  (t-\tau)} \norm{v_{1,\tau} - v_{2,\tau}  }  ^2 .
 \ee
So  the uniqueness and
 continuity of 
 solutions in  initial data  follow  
 immediately.
 
 Note that   \eqref{pexiv_5},
 \eqref{pexiv_20} and the uniqueness of solutions
 imply  that
 the entire  sequence $v_k(t_0, \tau, \omega, v_\tau)
 \to v(t_0, \tau, \omega, v_\tau)$ weakly in $\ltwo$
 for every fixed $t_0 \in [\tau, \tau +T]$
 and     $\omega \in \Omega$.
 By the  measurability  of 
  $v_k (t, \tau, \omega, v_\tau)$ in $\omega$,
  we obtain  the  measurability  of 
  $v (t, \tau, \omega, v_\tau)$ directly. 
 \end{proof}
 
 The following result is useful when
  proving the asymptotic compactness of solutions.

  \begin{lem}
 \label{comv1}
 Let   \eqref{f1}-\eqref{f3}  hold and
 $\{v_n\}_{n=1}^\infty$ be a bounded sequence
 in $\ltwo$.
 Then for every   
  $\tau  \in \R$,  $t >\tau$ and 
 $\omega \in \Omega$,
 there exist  $v_0 \in L^2(\tau, t;   \ltwo)$
 and  a subsequence 
 $\{v(\cdot, \tau, \omega, v_{n_m})\}_{m=1}^\infty$
  of
 $\{v(\cdot, \tau, \omega, v_n)\}_{n=1}^\infty$
 such that 
 $v(s, \tau, \omega, v_{n_m})
 \to v_0 (s) $ in $L^2(\o_k)$
 as $m \to \infty$
 for every  fixed $k \in \N$
 and for almost all $s \in (\tau, t)$.
 \end{lem}
 
 \begin{proof}
 Let $T$ be  a sufficiently large  number 
     such that
 $t \in (\tau, \tau +T]$. 
 Following the proof of \eqref{pexiv_6},  we 
 can show   that
    there  exists ${\widetilde{v}} \in L^2(\tau,  \tau +T;  \ltwo)$
 such   that,  up to  a subsequence,
 $$
  v(\cdot, \tau, \omega, v_{n} ) \to   {\widetilde{v}}
  \ \mbox{  strongly   in } \ L^2( \tau,  \tau+T;   L^2(\o_k ) )
  \quad \text{ for every }  k \in \N.
 $$
 Thus, for $k=1$,  there exist  a set $I_1 \subseteq [\tau, \tau +T]$
 of measure zero and a subsequence  
 $  v(\cdot, \tau, \omega, v_{n_1} )$ such that
 $$
  v(s, \tau, \omega, v_{n_1} ) \to   {\widetilde{v}} (s)
  \ \mbox{    in } \   L^2(\o_1)   
  \quad \mbox{    for  all }
 \    s \in [\tau,  \tau +T]\setminus I_1 .
 $$
  Similarly,  for $k=2$,  there
   exist  a set $I_2 \subseteq [\tau, \tau +T]$
 of measure zero and a subsequence  
 $  v(\cdot, \tau, \omega, v_{n_2} )$ of
 $  v(\cdot, \tau, \omega, v_{n_1} )$ 
   such that
$$
  v(s, \tau, \omega, v_{n_2} ) \to   {\widetilde{v}} (s)
  \ \mbox{    in } \   L^2(\o_2)   
  \quad \mbox{    for  all }
 \    s \in [\tau,  \tau +T]\setminus I_2 .
  $$
  Repeating this process we find that
  for each $k \in \N$, 
  there
   exist  a set $I_k \subseteq [\tau, \tau +T]$
 of measure zero and a subsequence  
 $  v(\cdot, \tau, \omega, v_{n_k} )$ of
 $  v(\cdot, \tau, \omega, v_{n_{k-1}} )$ 
   such that
 $$
  v(s, \tau, \omega, v_{n_k} ) \to   {\widetilde{v}}  (s)
  \ \mbox{    in } \   L^2(\o_k)   
  \quad \mbox{    for  all }
 \    s \in [\tau,  \tau +T]\setminus I_k .
 $$
  Let $I =\bigcup_{k=1}^\infty I_k$. Then
  by a diagonal process,  we infer that
  there exists a subsequence (which is
  still denoted by    
 $  v(\cdot, \tau, \omega, v_{n} )$)
   such that
 \be\label{pcomv1_1}
  v(s, \tau, \omega, v_{n} ) \to  {\widetilde{v}}  (s)
  \ \mbox{    in } \   L^2(\o_k)   
  \quad \mbox{    for  all }
 \    s \in [\tau,  \tau +T]\setminus I 
 \text{ and }   k   \in \N .
  \ee
  Note that $I$ has measure zero
  and   $t\in (\tau, \tau +T]$, which along with 
   \eqref{pcomv1_1}  completes  the proof.
 \end{proof}

  Based on Lemma \ref{exiv},  we  can
  define a continuous cocycle for
  problem  \eqref{seq1}-\eqref{seq2}
  in $\ltwo$.    
Let  $\Phi: \R^+ \times \R \times \Omega \times \ltwo$
$\to \ltwo$ be  a mapping  given by,
for every  $t \in \R^+$,  $\tau \in 
\R$, $\omega \in \Omega$  and $u_\tau \in \ltwo$,
 \be \label{pcycle}
 \Phi (t, \tau,  \omega, u_\tau)  
 =
  v(t+\tau, \tau,  \theta_{ -\tau} \omega,  v_\tau)
  + \eps h(x)   z(\theta_t \omega),
\ee 
where  $v$  is the solution of
system \eqref{veq1}-\eqref{veq2}
with initial condition $v_\tau =    u_\tau - \eps h(x) z(\omega) $
at initial time $\tau$. 
Note that \eqref{uv}  and \eqref{pcycle}
imply 
\be\label{pcycle2}
\Phi (t, \tau,  \omega, u_\tau)  
 =
  u(t+\tau, \tau,  \theta_{ -\tau} \omega,  u_\tau),
  \ee
  where $u$ is a solution of \eqref{seq1}-\eqref{seq2}
  in  some sense.
  Since   the solution $v$ of \eqref{veq1}-\eqref{veq2}
  is $(\calf, \calb(\ltwo))$-measurable in $\omega$
  and continuous in initial data in $\ltwo$, we find that
  $\Phi (t, \tau,  \omega, u_\tau)  $ given by \eqref{pcycle} 
  is also $(\calf, \calb(\ltwo))$-measurable in $\omega$
  and continuous in $u_\tau$ in $\ltwo$.
  In fact,  one can verify  that $\Phi$
  is a continuous cocycle  on $\ltwo$
over 
$(\Omega, \calf, P,  \{\theta\}_{t \in \R})$
in the sense that 
   for all
  $\tau\in \R$,
  $\omega \in   \Omega $
  and    $t, s \in \R^+$,  
\begin{itemize}
\item [(i)]   $\Phi (\cdot, \tau, \cdot, \cdot): \R ^+ \times \Omega \times \ltwo
\to   \ltwo$ is
 $(\calb (\R^+)   \times \calf \times \calb (\ltwo),  
\calb(\ltwo))$-measurable.

\item[(ii)]    $\Phi(0, \tau, \omega, \cdot) $ is the identity map on $ \ltwo$.

\item[(iii)]    $\Phi(t+ s, \tau, \omega, \cdot) =
 \Phi(t,  \tau +s,  \theta_{s} \omega, \cdot) 
 \circ \Phi(s, \tau, \omega, \cdot)$.

\item[(iv)]    $\Phi(t, \tau, \omega,  \cdot):  \ltwo \to   \ltwo$
 is continuous.
    \end{itemize}
    Note that   the cocycle property (iii) of $\Phi$ can be easily proved
    by \eqref{pcycle} and the properties of the solution $v$
    of the pathwise deterministic equation \eqref{veq1}-\eqref{veq2}.
    Our goal is to   establish
     the existence of random    attractors  of $\Phi$
     with 
   an appropriate  attraction domain.
   To specify such an attraction domain, we  consider
    a family  
  $D =\{ D(\tau, \omega) \subseteq \ltwo: \tau \in \R, \omega \in \Omega \}$ 
  of   
  bounded nonempty    sets  such that
  for every $\tau \in \R$   and $\omega \in \Omega$, 
 \be
 \label{Dom1}
 \lim_{s \to  - \infty} e^{ {\frac 54} \lambda s
 + 2\alpha  \int^0_s   \eta (\theta_r \omega) dr }
  \| D( \tau + s, \theta_{s} \omega ) \|^2 =0,
 \ee 
   where 
   $\| S \|=  \sup\limits_{u \in S}
   \| u\|_{\ltwo }$ for a nonempty bounded  subset $S$ of $\ltwo$.
  In  the sequel,   we will   use  $\cald$
   to denote the collection of all families 
   with property \eqref{Dom1}:
 \be
 \label{Dom2}
\cald  = \{ 
   D =\{ D(\tau, \omega) \subseteq \ltwo:
    \tau \in \R, \omega \in \Omega \}: \ 
 D  \ \mbox{satisfies} \  \eqref{Dom1} \} .
\ee
We   will construct a $\cald$-pullback attractor
$\cala  = \{\cala (\tau, \omega): \tau \in \R,
  \omega \in \Omega \} \in  \cald$ 
  for  $\Phi$ in $\ltwo$ in the following sense : 
 \begin{itemize}
\item [(i)]   $ {\mathcal{A}}$ is measurable
  and
 $ {\mathcal{A}}(\tau, \omega)$ is compact for all $\tau \in \R$
and    $\omega \in \Omega$.

\item[(ii)]   $ {\mathcal{A}}$  is invariant, that is,
for every $\tau \in \R$ and
 $\omega \in \Omega$,
$$ \Phi(t, \tau, \omega,  {\mathcal{A}}(\tau, \omega)   )
=  {\mathcal{A}} ( \tau +t, \theta_{t} \omega
), \ \  \forall \   t \ge 0.
$$

\item[(iii)]   For every
 $B = \{B(\tau, \omega): \tau \in \R, \omega \in \Omega\}
 \in  {\mathcal{D}}$ and for every $\tau \in \R$ and
 $\omega \in \Omega$,
$$ \lim_{t \to  \infty} \text{dist}_{\ltwo} (\Phi(t, \tau -t,
 \theta_{-t}\omega, B(\tau -t, 
 \theta_{-t}\omega) ) ,  {\mathcal{A}} (\tau, \omega ))=0,
$$
where  $\text{dist}_{\ltwo}$    is the Hausdorff   semi-distance
between two sets in $\ltwo$.
 \end{itemize}
 
 We  will apply  the following result from 
   \cite{wan5} to show  the existence of 
   $\cald$-pullback attractors   for $\Phi$.
  Similar results on existence of random attractors
    can be found in 
\cite{bat1, car6,  cra2, fla1,  gess2, schm1}.

\begin{prop}
\label{att}  
 Let $ \cald$ be   the 
 collection given by \eqref{Dom2}.
If   
$\Phi$ is $ \cald$-pullback asymptotically
compact in $\ltwo$ and $\Phi$ has a  
closed  
   measurable  
     $ \cald$-pullback absorbing set
  $K$ in $ \cald$,  then
$\Phi$ has a unique  $ \cald$-pullback
attractor $ \cala $  in $\ltwo$  which is  given  by, 
for each $\tau  \in \R$   and
$\omega \in \Omega$,
$$
 \cala (\tau, \omega)
=\Omega(K, \tau, \omega)
=\bigcup_{B \in  {\mathcal{D}}} \Omega(B, \tau, \omega),
$$
where $\Omega(K)
=\{ \Omega(K, \tau, \omega): \tau \in \R,
\omega \in \Omega\}$
is the $\omega$-limit set of $K$.

If, in addition, 
there is $T>0$  such
that  
$\Phi (t, \tau +T, \omega, \cdot)
=\Phi (t, \tau, \omega, \cdot)$
and $K(\tau +T, \omega) = K(\tau, \omega)$
for   all  $t\in \R^+$,  $\tau \in \R$
and $\omega \in \Omega$,
then the attractor $\cala$ is pathwise $T$-periodic, i.e.,
$\cala(\tau +T, \omega) = \cala(\tau, \omega)$
for   all     $\tau \in \R$
and $\omega \in \Omega$. 
  \end{prop}

We remark that  the $\calf$-measurability of the
   attractor   $ {\mathcal{A}}$  was  given   in  
\cite{wan7} and the  measurability of $ {\mathcal{A}}$
   with respect to 
   the $P$-completion
 of $\calf$   was given  in \cite{wan5}.
For our purpose,   we 
further assume  the following condition
on $g$, $\psi_1$  and $\psi_3$:
for every $ \tau \in \R$, 
 \be\label{g1}
\int_{-\infty}^\tau e^{\lambda s}
\left (
\| g(s, \cdot) \|^2 + \| \psi_1 (s, \cdot) \|_{L^1({\R^n})}
+ \| \psi_3 (s, \cdot) \|_{L^{q_1}({\R^n})}^{q_1}
\right ) ds < \infty.
\ee

\section{Uniform Estimates of Solutions}
 \setcounter{equation}{0}
 
 This section is devoted to 
  uniform estimates of  solutions   of 
  \eqref{seq1}  and \eqref{veq1}
  which are needed for proving the
  existence of random attractors for $\Phi$.
  When deriving uniform estimates,  
  the following     positive number
  $\alpha_0$ is useful:
 \be
 \label{alphazero}
\alpha_0 = \frac{1}{8(1+\abs{E(\eta)})}\lambda.
\ee

\begin{lem}
\label{est1}
Let $\alpha_0$ be the  positive number given by \eqref{alphazero}.
 Suppose  \eqref{f1}-\eqref{f3}  and \eqref{g1} hold.
Then for every  $\alpha \le \alpha_0$,
 $\sigma \in \R$,
 $\tau \in \R$, $\omega \in \Omega$   and $D=\{D(\tau, \omega)
: \tau \in \R,  \omega \in \Omega\}  \in \cald$,
 there exists  $T=T(\tau, \omega,  D, \sigma, \alpha)>0$ such that 
 for all $t \ge T$,  the solution
 $v$ of  problem  \eqref{veq1}-\eqref{veq2}     satisfies 
 $$
\norm{v(\sigma, \tau-t,\theta_{-\tau}\omega,v_{\tau-t})}^2 
+
\int_{\tau-t}^{\sigma}e^{\frac{5}{4}\lambda(s-\sigma)
-2\alpha\int_{\sigma-\tau}^{s-\tau}
\eta(\theta_{r}\omega)dr}
\norm{v(s, \tau-t,\theta_{-\tau}\omega,v_{\tau-t})}^2ds 
$$
$$
+
\int_{\tau-t}^{\sigma}e^{\frac{5}{4}
\lambda(s-\sigma)-2\alpha\int_{\sigma-\tau}
^{s-\tau}\eta(\theta_{r}\omega)dr}
\norm{\nabla (v(s, \tau-t,\theta_{-\tau}\omega,v_{\tau-t})
+\eps h(x)z(\theta_{s-\tau} \omega))}^p_pds 
$$
$$
+
\int_{\tau-t}^{\sigma}e^{\frac{5}{4}
\lambda(s-\sigma)-2\alpha\int_{\sigma-\tau}
^{s-\tau}\eta(\theta_{r}\omega)dr}
\norm{v(s, \tau-t,\theta_{-\tau}\omega,
v_{\tau-t})+\eps h(x)z(\theta_{s-\tau} \omega))}^q_qds
\le M,
$$
where $v_{\tau -t}  \in D(\tau -t, \theta_{ -t} \omega)$
  and
  $M=M(\tau, \omega, \sigma, \alpha, \eps) $ is given by
$$
M
= c\int_{-\infty}^{\sigma-\tau}e^{\frac{5}{4}
\lambda(s-\sigma+\tau)-2\alpha\int_{\sigma-\tau}^{s}
\eta(\theta_{r}\omega)dr} 
 \left ( \abs{\eps z(\theta_{s} \omega)}^p 
+ \abs{\eps z(\theta_{s} \omega)}^q 
+ \abs{\alpha \eps \eta(\theta_{s}\omega)z(\theta_{s}\omega )
}^2 \right )ds
$$
$$
+ c\int_{-\infty}^{\sigma-\tau}e^{\frac{5}{4}
\lambda(s-\sigma+\tau)-2\alpha
\int_{\sigma-\tau}^{s}\eta(\theta_{r}\omega)dr}
(\norm{g(s+\tau)}^2 
+ \norm{\psi_1(s+\tau)}_{1}
+\norm{\psi_3(s+\tau)}_{q_1}^{q_1})ds,
$$
with 
  $c$  being  a  positive constant  
   independent of $\tau$, $\omega$,  $D$,
  $\alpha$ and $\eps$.
\end{lem}

\begin{proof}  
Using energy equation \eqref{enlem1}
and 
following the proof of \eqref{plvk1_6},
 we  obtain
  \be
  \label{pest1_1}
\frac{d}{dt}\norm{v}^2 
+ {\frac{7}{4}}  \lambda \norm{v}^2 
+ \ii \abs{\nabla (v+\eps hz(\theta_{t} \omega))}^p_p  dx
+ \gamma
\ii \abs{v+\eps h z(\theta_{t} \omega))}^q_q dx
$$
$$
\leq 2\alpha\eta(\theta_t\omega)\norm{v}^2 
+ c_3 \left (  \abs{\eps z(\theta_{t} \omega)}^p 
+   \abs{\eps z(\theta_{t} \omega)}^q 
+     \abs{\alpha \eps \eta(\theta_t\omega)z(\theta_t\omega) }^2
\right )
$$
$$
 + c_4 \left (
 \norm{g(t)}^2\
 + \norm{\psi_1(t)}_{1}+   \norm{\psi_3(t)}_{q_1}^{q_1}
 \right ) .
\ee
Multiplying \eqref{pest1_1} 
by $e^{\frac{5}{4}\lambda t-2\alpha\int^{t}_{0}
\eta(\theta_{r}\omega)dr}$, and
then  integrating from $\tau -t$ to $\sigma$
with $\sigma \ge \tau -t$, 
we get,  
 \be \label{pest1_2}
\norm{v(\sigma, \tau-t,\omega,v_{\tau-t})}^2
 + \frac{\lambda}{2} \int_{\tau-t}^{\sigma}
 e^{\frac{5}{4}\lambda(s-\sigma)
 -2\alpha\int_{\sigma}^{s}
 \eta(\theta_{r}\omega)dr}
 \norm{v(s, \tau-t,\omega,v_{\tau-t})}^2ds
 $$
 $$
+ \int_{\tau-t}^{\sigma}e^{\frac{5}{4}
\lambda(s-\sigma)-2\alpha\int_{\sigma}^{s}
\eta(\theta_{r}\omega)dr}
\norm{\nabla (v(s, \tau-t,\omega,v_{\tau-t})
+\eps hz(\theta_{s} \omega))}^p_pds
$$
$$
+ \gamma\int_{\tau-t}^{\sigma}
e^{\frac{5}{4}\lambda(s-\sigma)
-2\alpha\int_{\sigma}^{s}\eta(\theta_{r}\omega)dr}
\norm{v(s, \tau-t,\omega,v_{\tau-t})
+\eps h z(\theta_{s} \omega) }^q_qds
$$
$$
\leq e^{\frac{5}{4}\lambda(\tau-t-\sigma)
-2\alpha\int_{\sigma}^{\tau-t}\eta(\theta_{r}\omega)dr}
\norm{v_{\tau-t}}^2  
$$
$$
+ c_3 \int_{\tau-t}^{\sigma}e^{\frac{5}{4}
\lambda(s-\sigma)-2\alpha\int_{\sigma}^{s}
\eta(\theta_{r}\omega)dr}( \abs{\eps z(\theta_{s} \omega)}^p
 +  \abs{\eps z(\theta_{s} \omega)}^q 
 +  \abs{\alpha \eps \eta(\theta_{s}\omega)z(\theta_s\omega)}^2
 )ds
 $$
 $$
+ c_4 \int_{\tau-t}^{\sigma}
e^{\frac{5}{4}\lambda(s-\sigma)
-2\alpha\int_{\sigma}^{s}\eta(\theta_{r}\omega)dr}
 \left (
 \norm{g(s)}^2
 + 
 \norm{\psi_1(s)}_{1}+\norm{\psi_3(s)}_{q_1}^{q_1} 
 \right )
 ds.
\ee
 Replacing $\omega$ with $\theta_{-\tau}\omega$
  in \eqref{pest1_2},  we get
 \be
 \label{pest1_4}
\norm{v(\sigma, \tau-t,\theta_{-\tau}\omega,v_{\tau-t})}^2 
+ \frac{\lambda}{2} \int_{\tau-t}^{\sigma}
e^{\frac{5}{4}\lambda(s-\sigma)
- 2\alpha\int_{\sigma-\tau}^{s-\tau}\eta(\theta_{r}\omega)dr}
\norm{v(s, \tau-t,\theta_{-\tau}\omega,v_{\tau-t})}^2ds
$$
$$
+ \int_{\tau-t}^{\sigma}e^{\frac{5}{4}\lambda(s-\sigma)
-2\alpha\int_{\sigma-\tau}^{s-\tau}\eta(\theta_{r}\omega)dr}
\norm{\nabla (v(s, \tau-t,\theta_{-\tau}\omega,v_{\tau-t})
+\eps h z(\theta_{s-\tau} \omega))}^p_pds
$$
$$
+ \gamma\int_{\tau-t}^{\sigma}e^{\frac{5}{4}
\lambda(s-\sigma)-2\alpha\int_{\sigma-\tau}^{s-\tau}
\eta(\theta_{r}\omega)dr}
\norm{v(s, \tau-t,\theta_{-\tau}\omega,v_{\tau-t})
+\eps h z(\theta_{s-\tau} \omega)}^q_qds
$$
$$
\leq e^{\frac{5}{4}\lambda
 (\tau -t -\sigma)+2\alpha\int^{\sigma -\tau}_{-t}
\eta(\theta_{r}\omega)dr}  \norm{v_{\tau-t}}^2
$$
$$
+ c_3 \int_{-t}^{\sigma-\tau}e^{\frac{5}{4}
\lambda(s +\tau -\sigma)-2\alpha\int_{\sigma -\tau}^{s}
\eta(\theta_{r}\omega)dr}( \abs{\eps z(\theta_{s} \omega)}^p
 +  \abs{\eps z(\theta_{s} \omega)}^q 
 +  \abs{\alpha \eps \eta(\theta_{s}\omega)z(\theta_s\omega)}^2
 )ds
 $$
 $$
+ c_4 
 \int_{-t}^{\sigma-\tau}e^{\frac{5}{4}
\lambda(s +\tau -\sigma)-2\alpha\int_{\sigma -\tau}^{s}
\eta(\theta_{r}\omega)dr}
 \left (
 \norm{g(s+\tau)}^2
 + 
 \norm{\psi_1(s+\tau)}_{1}+\norm{\psi_3(s+\tau)}_{q_1}^{q_1} 
 \right )
 ds.
 \ee
 By the ergodicity of $\eta$, \eqref{alphazero}  and \eqref{g1}
 one can verify  that   for all $\alpha \le \alpha_0$, 
 \be\label{pest1_7}
  \int_{-\infty}^{\sigma-\tau}e^{\frac{5}{4}
\lambda(s +\tau -\sigma)-2\alpha\int_{\sigma -\tau}^{s}
\eta(\theta_{r}\omega)dr}
 \left (
 \norm{g(s+\tau)}^2
 + 
 \norm{\psi_1(s+\tau)}_{1}+\norm{\psi_3(s+\tau)}_{q_1}^{q_1} 
 \right )
 ds <\infty.
 \ee
 Similarly, by the temperedness of $\eta$  and  $z$, we can prove  that
 for all  $\alpha \le \alpha_0$,
 \be\label{pest1_9}
 \int_{-\infty}^{\sigma-\tau}e^{\frac{5}{4}
\lambda(s +\tau -\sigma)-2\alpha\int_{\sigma -\tau}^{s}
\eta(\theta_{r}\omega)dr}( \abs{\eps z(\theta_{s} \omega)}^p
 +  \abs{\eps z(\theta_{s} \omega)}^q 
 +  \abs{\alpha \eps \eta(\theta_{s}\omega)z(\theta_s\omega)}^2
 )ds <\infty.
 \ee
 Since $v_{\tau -t} \in    D(\tau-t, \theta_{-t}\omega)$ 
 and $D \in \cald$, by \eqref{Dom1}-\eqref{Dom2} we obtain
 $$
 e^{\frac{5}{4}\lambda
 (\tau -t -\sigma)+2\alpha\int^{\sigma -\tau}_{-t}
\eta(\theta_{r}\omega)dr}  \norm{v_{\tau-t}}^2
 $$
  $$
 \le e^{\frac{5}{4}\lambda
 (\tau  -\sigma)+2\alpha\int^{\sigma -\tau}_{0}\eta(\theta_{r}\omega)dr}
 e^{- \frac{5}{4}\lambda t 
 +2\alpha\int^{0}_{-t}
\eta(\theta_{r}\omega)dr}
  \norm{D(\tau -t, \theta_{-t} \omega)}^2
  \to 0,
 $$
 as $t \to \infty$.
 Therefore, 
  there exists $T = T(\tau, \omega,  D, \sigma, \alpha)>0$ 
  such that for all $t \geq T$, 
  $$
 e^{\frac{5}{4}\lambda
 (\tau -t -\sigma)+2\alpha\int^{\sigma -\tau}_{-t}
\eta(\theta_{r}\omega)dr}  \norm{v_{\tau-t}}^2
$$
$$
\le
 \int_{-\infty}^{\sigma-\tau}e^{\frac{5}{4}
\lambda(s +\tau -\sigma)-2\alpha\int_{\sigma -\tau}^{s}
\eta(\theta_{r}\omega)dr}
 \left (
 \norm{g(s+\tau)}^2
 + 
 \norm{\psi_1(s+\tau)}_{1}+\norm{\psi_3(s+\tau)}_{q_1}^{q_1} 
 \right )
 ds ,
 $$
 which along with
 \eqref{pest1_4}-\eqref{pest1_9}
 concludes    the proof.
  \end{proof}

By   Lemma \ref{est1},  we  obtain  the following estimates.

\begin{lem}
\label{est2} 
 Suppose  \eqref{f1}-\eqref{f3}  and \eqref{g1} hold.
Then for every $\alpha \le \alpha_0$, 
 $\tau \in \R$, $\omega \in \Omega$   and $D=\{D(\tau, \omega)
: \tau \in \R,  \omega \in \Omega\}  \in \cald$,
 there exists  $T=T(\tau, \omega,  D, \alpha)>0$ such that 
 for all $t \ge T$  and ,  the solution
 $v$ of  problem  \eqref{veq1}-\eqref{veq2}     satisfies 
 $$
\norm{v(\tau, \tau-t,\theta_{-\tau}\omega,v_{\tau-t})}^2 
+
\int_{\tau-t}^{\tau}e^{\frac{5}{4}\lambda(s-\tau)
-2\alpha\int_{0}^{s-\tau}
\eta(\theta_{r}\omega)dr}
\norm{v(s, \tau-t,\theta_{-\tau}\omega,v_{\tau-t})}^2ds 
$$
$$
+
\int_{\tau-t}^{\tau}e^{\frac{5}{4}
\lambda(s-\tau)-2\alpha\int_{0}
^{s-\tau}\eta(\theta_{r}\omega)dr}
\norm{\nabla (v(s, \tau-t,\theta_{-\tau}\omega,v_{\tau-t})
+\eps h(x)z(\theta_{s-\tau} \omega))}^p_pds 
$$
\be\label{est2_1}
+
\int_{\tau-t}^{\tau}e^{\frac{5}{4}
\lambda(s-\tau)-2\alpha\int_{0}
^{s-\tau}\eta(\theta_{r}\omega)dr}
\norm{v(s, \tau-t,\theta_{-\tau}\omega,
v_{\tau-t})+\eps h(x)z(\theta_{s-\tau} \omega))}^q_qds
\le R(\tau, \omega, \alpha, \eps),
\ee
where $v_{\tau -t}  \in D(\tau -t, \theta_{ -t} \omega)$
  and
  $R(\tau, \omega,  \alpha, \eps) $ is given by
$$
 R(\tau, \omega,  \alpha, \eps)
= c\int_{-\infty}^{0}e^{\frac{5}{4}
\lambda s -2\alpha\int_{0}^{s}
\eta(\theta_{r}\omega)dr} 
 \left ( \abs{\eps z(\theta_{s} \omega)}^p 
+ \abs{\eps z(\theta_{s} \omega)}^q 
+ \abs{\alpha \eps \eta(\theta_{s}\omega)z(\theta_{s}\omega )
}^2 \right )ds
$$
\be\label{est2_2}
+c\int_{-\infty}^{0}e^{\frac{5}{4}
\lambda s -2\alpha\int_{0}^{s}
\eta(\theta_{r}\omega)dr}  
(\norm{g(s+\tau)}^2 
+ \norm{\psi_1(s+\tau)}_{1}
+\norm{\psi_3(s+\tau)}_{q_1}^{q_1})ds,
\ee
 with 
  $c$  being  a  positive constant   independent of $\tau$, $\omega$,  $D$,
  $\alpha$ and $\eps$.
  In addition,   we have
   \be\label{est2_3}
  \lim_{t \to \infty}
  e^{-{\frac 54} \lambda t + 2 \alpha \int_{-t}^0 \eta (\theta_r\omega) dr}  
    R(\tau- t, \theta_{-t} \omega, \alpha, \eps) =0.
  \ee
\end{lem}

\begin{proof}
\eqref{est2_1}  and \eqref{est2_2}
are special cases  of  Lemma \ref{est1}
for   $\sigma =\tau$.  We now  prove \eqref{est2_3}.
 By  \eqref{est2_2}  we  have
$$  R(\tau- t, \theta_{-t} \omega, \alpha, \eps)
$$
$$
 =
  c\int_{-\infty}^{0}e^{\frac{5}{4}
\lambda s -2\alpha\int_{0}^{s}
\eta(\theta_{r-t}\omega)dr} 
 \left ( \abs{\eps z(\theta_{s-t} \omega)}^p 
+ \abs{\eps z(\theta_{s-t} \omega)}^q 
+ \abs{\alpha \eps \eta(\theta_{s-t}\omega)z(\theta_{s-t}\omega )
}^2 \right )ds
$$
$$
+c\int_{-\infty}^{0}e^{\frac{5}{4}
\lambda s -2\alpha\int_{0}^{s}
\eta(\theta_{r-t}\omega)dr}  
(\norm{g(s+\tau -t)}^2 
+ \norm{\psi_1(s+\tau-t)}_{1}
+\norm{\psi_3(s+\tau-t)}_{q_1}^{q_1})ds
$$
$$
 =
  c\int_{-\infty}^{-t}e^{\frac{5}{4}
\lambda (t+s) -2\alpha\int_{-t}^{s}
\eta(\theta_{r}\omega)dr} 
 \left ( \abs{\eps z(\theta_{s} \omega)}^p 
+ \abs{\eps z(\theta_{s} \omega)}^q 
+ \abs{\alpha \eps \eta(\theta_{s}\omega)z(\theta_{s}\omega )
}^2 \right )ds
$$
$$
+c\int_{-\infty}^{-t}e^{\frac{5}{4}
\lambda (t+s) -2\alpha\int_{-t}^{s}
\eta(\theta_{r}\omega)dr}  
(\norm{g(s+\tau )}^2 
+ \norm{\psi_1(s+\tau)}_{1}
+\norm{\psi_3(s+\tau)}_{q_1}^{q_1})ds.
$$
Therefore we get
$$
  e^{-{\frac 54} \lambda t + 2 \alpha \int_{-t}^0 \eta (\theta_r\omega) dr}  
    R(\tau- t, \theta_{-t} \omega, \alpha, \eps) 
    $$
    $$
 =
  c\int_{-\infty}^{-t}e^{\frac{5}{4}
\lambda s  -2\alpha\int_{0}^{s}
\eta(\theta_{r}\omega)dr} 
 \left ( \abs{\eps z(\theta_{s} \omega)}^p 
+ \abs{\eps z(\theta_{s} \omega)}^q 
+ \abs{\alpha \eps \eta(\theta_{s}\omega)z(\theta_{s}\omega )
}^2 \right )ds
$$
\be\label{pest2_1}
+c\int_{-\infty}^{-t}e^{\frac{5}{4}
\lambda  s  -2\alpha\int_{0}^{s}
\eta(\theta_{r}\omega)dr}  
(\norm{g(s+\tau )}^2 
+ \norm{\psi_1(s+\tau)}_{1}
+\norm{\psi_3(s+\tau)}_{q_1}^{q_1})ds.
\ee
Since the integrals in \eqref{est2_2} are convergent, 
by \eqref{pest2_1}
we obtain 
$ e^{-{\frac 54} \lambda t + 2 \alpha \int_{-t}^0 \eta (\theta_r\omega) dr}  
    R(\tau- t, \theta_{-t} \omega, \alpha, \eps) 
    \to 0$  as $t \to \infty$.
    This completes   the proof.
\end{proof}

Next,  we derive uniform estimates on the
tails of 
 solutions of \eqref{veq1}-\eqref{veq2}
 outside a bounded domain. These estimates  are
 crucial for proving the asymptotic compactness of solutions
 on unbounded domains.

 \begin{lem}
 \label{est3}
Suppose \eqref{f1}-\eqref{f3} and \eqref{g1} hold. 
Then for every $\nu > 0$,  $\alpha \le \alpha_0$,
$\eps>0$, 
 $\tau \in \R$,
$\omega \in \Omega$  and  $D\in \cald$, 
there exists $T = T(\tau, \omega, D,  \alpha, \eps,  \nu) > 0$
 and $K = K(\tau, \omega,  \alpha, \eps,  \nu) \geq 1$
  such that for all $t \geq T$
  and $\sigma \in [\tau -1, \tau]$, 
  the solution $v$ of  
  \eqref{veq1}-\eqref{veq2}   satisfies
$$
\int_{\abs{x}\geq K}
\abs{v(\sigma,  \tau-t,\theta_{-\tau}\omega,v_{\tau-t})}^2dx
 \leq \nu,
$$
where $v_{\tau -t} \in D(\tau -t, \theta_{-t} \omega )$.
In addition,  $  T(\tau, \omega, D,  \alpha, \eps,  \nu) $
and $  K(\tau, \omega, D,  \alpha, \eps,  \nu) $
are  uniform  with respect to  $\eps \in (0,1]  $.
\end{lem}

\begin{proof}
Let $\rho$ be a smooth
 function defined on $\R^{+}$ 
 such that $0\leq \rho(s) \leq 1$ for
  all $s \in \R^{+}$, and 
\begin{equation*} 
   \rho(s) = \left\{
     \begin{array}{ll}
       0  &  \text{ for } 0\leq s \leq 1;\\
       1  &  \text{ for } s \geq 2.
     \end{array}
   \right.
\end{equation*} 
 Multiplying \eqref{veq1}
  by $\rho(\frac{\abs{x}^2}{k^2})v$ and 
  then integrating over $\R^n$ we  get
 \be \label{pest3_1}
\frac{1}{2}\frac{d}{dt}\int_{\R^n}
 \rho(\frac{\abs{x}^2}{k^2})\abs{v}^2dx 
 - \int_{\R^n}
 \rho(\frac{\abs{x}^2}{k^2})\text{div}
 (\abs{\nabla (v+\eps h z(\theta_{t} \omega))}
 ^{p-2}\nabla (v+\eps h z(\theta_{t} \omega)))vdx
 $$
 $$
= ( \alpha\eta(\theta_t\omega)- \lambda )
 \ii \rho(\frac{\abs{x}^2}{k^2})\abs{v}^2dx
+ \int_{R^n}\rho(\frac{\abs{x}^2}{k^2})
 f(t,x,v+\eps h z(\theta_{t} \omega)) v dx
 $$ 
 $$
 + \alpha \eps \eta (\theta_t\omega) z(\theta_t\omega) 
  \ii \rho(\frac{\abs{x}^2}{k^2}) hvdx 
+ \ii \rho(\frac{\abs{x}^2}{k^2})g(t,x)vdx.
\ee
For  the term involving the divergence  we have
 \be \label{pest3_3}
\int_{\R^n}\rho(\frac{\abs{x}^2}{k^2})
\text{div}(\abs{\nabla (v+\eps hz(\theta_{t} \omega))}
^{p-2}\nabla (v+\eps hz(\theta_{t} \omega)))vdx
$$
$$
= -\int_{\R^n}\rho(\frac{\abs{x}^2}{k^2})
\abs{\nabla (v+\eps hz(\theta_{t} \omega))}^{p}dx 
$$
$$
+ \int_{\R^n}\rho(\frac{\abs{x}^2}{k^2})
\abs{\nabla (v+\eps hz(\theta_{t} \omega))}
^{p-2} \nabla (v+\eps hz(\theta_{t} \omega))
\cdot\nabla (\eps hz(\theta_{t} \omega )) dx 
$$
$$
-\int_{\R^n}\rho'(\frac{\abs{x}^2}{k^2}) \frac{2x}{k^2}
\cdot\nabla (v+\eps hz(\theta_{t} \omega)) 
\abs{\nabla (v+\eps hz(\theta_{t} \omega))}^{p-2}vdx 
$$
$$
\leq 
- {\frac 12} \int_{\R^n}\rho(\frac{\abs{x}^2}{k^2})
\abs{\nabla (v+\eps hz(\theta_{t} \omega))}^{p}dx 
+c_1 \int_{\R^n}\rho(\frac{\abs{x}^2}{k^2})
\abs{\nabla(\eps hz(\theta_{t} \omega))}^{p}dx
$$
$$
-\int_{k \leq \abs{x} \leq 2k}
\rho'(\frac{\abs{x}^2}{k^2})
 \frac{2x}{k^2}\cdot\nabla 
(v+\eps hz(\theta_{t} \omega)) 
\abs{\nabla (v+\eps hz(\theta_{t}
 \omega))}^{p-2}vdx 
 $$
 $$
\leq c_1 \int_{\R^n}\rho(\frac{\abs{x}^2}
{k^2})\abs{\nabla(\eps hz(\theta_{t} \omega))}
^{p}dx + \frac{c_2}{k}(\norm{v}
^p_p+\norm{\nabla(v
+\eps hz(\theta_{t} \omega))}^p_p).
\ee
As in \eqref{plvk1_3},   for the nonlinearity $f$   we have
  \be 
  \label{pest3_5}
\int_{R^n}\rho(\frac{\abs{x}^2}{k^2})
f(t,x,v+\eps h z(\theta_{t} \omega))vdx 
 $$
 $$
\leq \int_{R^n}\rho(\frac{\abs{x}^2}{k^2})
f(t,x,v+\eps h z(\theta_{t} \omega))
(v+\eps h z(\theta_{t} \omega))dx 
+ \int_{\R^n}\rho(\frac{\abs{x}^2}{k^2})
\abs{f(t,x,v+\eps h z(\theta_{t} \omega))} \;
\abs{\eps h z(\theta_{t} \omega)}dx
$$
$$
\leq -\gamma\int_{\R^n}\rho(\frac{\abs{x}^2}
{k^2})\abs{v+\eps h z(\theta_{t} \omega))}^q dx 
+ \int_{\R^n}\rho(\frac{\abs{x}^2}{k^2})\psi_1(t,x)dx
$$
$$
 +
\int_{\R^n}\rho(\frac{\abs{x}^2}{k^2})
\psi_2(t,x)\abs{v+\eps h z(\theta_{t} \omega))}
^{q-1}\abs{\eps h z(\theta_{t} \omega)}dx 
+ \int_{\R^n}\rho(\frac{\abs{x}^2}{k^2})
\abs{\psi_3(t,x)\eps h z(\theta_{t} \omega)}dx
$$
$$
\leq -\frac{\gamma}{2}\int_{\R^n}
\rho(\frac{\abs{x}^2}{k^2})
\abs{v+\eps h z(\theta_{t} \omega))}^q dx
 + \int_{\R^n}\rho(\frac{\abs{x}^2}{k^2})
 (\abs{\psi_1(t,x)} + \abs{\psi_3(t,x)}^{q_1})dx 
 $$
 $$
 + c_3 \int_{\R^n}\rho(\frac{\abs{x}^2}{k^2})
 \abs{\eps h z(\theta_{t} \omega)}^q dx.
\ee
Note that
\be 
\label{pest3_7}
\alpha \eps \eta(\theta_t\omega)z(\theta_t\omega)
\int_{R^n}\rho(\frac{\abs{x}^2}{k^2})
hvdx + \int_{R^n}\rho(\frac{\abs{x}^2}{k^2})g(t,x)vdx
$$
$$
\le
  c_4 \int_{R^n}\rho(\frac{\abs{x}^2}{k^2})
  \abs{\alpha \eps \eta(\theta_t\omega)
  z(\theta_t\omega) h} ^2 dx
  $$
  $$
+ c_5
\int_{R^n}\rho(\frac{\abs{x}^2}{k^2})
\abs{g (t,x)}^2dx 
+ \frac{3}{8}\lambda\int_{R^n}
\rho(\frac{\abs{x}^2}{k^2})\abs{v}^2dx.
\ee
It follows   from \eqref{pest3_1}-\eqref{pest3_7} that 
 \be
 \label{pest3_11} 
\frac{d}{dt}\int \rho(\frac{\abs{x}^2}{k^2})\abs{v}^2dx
 + 
   (
 \frac{5}{4}\lambda -2\alpha \eta  (\theta_t\omega)
   )
 \int \rho(\frac{\abs{x}^2}{k^2})
 \abs{v}^2dx 
 $$
 $$
 \leq  
    \frac{c_7 }{k}(\norm{v}^p_p
  +\norm{\nabla(v+\eps hz(\theta_{t} \omega))}^p_p)
  $$
  $$
+ c_7 \int_{\R^n}\rho(\frac{\abs{x}^2}{k^2})
( \abs{g(t,x)}^2+ \abs{\psi_1(t,x)} + \abs{\psi_3(t,x)}^{q_1})dx
$$
$$
 + c_7 \int_{\R^n}\rho(\frac{\abs{x}^2}{k^2})
 (
  \abs{\nabla \eps h z(\theta_{t} \omega)}^p 
  +
 \abs{\eps h z(\theta_{t} \omega)}^q 
 +
 \abs{\alpha \eps \eta  (\theta_{t} \omega)  z(\theta_{t} \omega) h}^2 
 )dx. 
\ee
 Since $h \in H^1(\R^n)
 \bigcap 
  W^{1,q}(\R^n)$
  with $2\le p\le q$, 
  we  find  that for every $\nu > 0$, there exists a $
  K_1=K_1(\nu) \ge 1$ such that for all $k  \ge  K_1$,
 \be
 \label{pest3_13}
c_7 \int_{\R^n}\rho(\frac{\abs{x}^2}{k^2})
 (
  \abs{\nabla \eps h z(\theta_{t} \omega)}^p 
  +
 \abs{\eps h z(\theta_{t} \omega)}^q 
 +
 \abs{\alpha \eps \eta  (\theta_{t} \omega)  z(\theta_{t} \omega) h}^2
 )dx
 $$
 $$
 =c_7 
 \int_{\abs{x} \geq k}
 \rho(\frac{\abs{x}^2}{k^2})
 (
  \abs{\nabla \eps h z(\theta_{t} \omega)}^p 
  +
 \abs{\eps h z(\theta_{t} \omega)}^q 
 +
 \abs{\alpha \eps \eta  (\theta_{t} \omega)  z(\theta_{t} \omega) h}^2
 )dx
 $$
 $$
  \leq \nu
  \left (
    \abs{\eps z(\theta_t\omega)}^p
    + \abs{\eps z(\theta_t\omega)}^q
    + \abs{\alpha \eps \eta  (\theta_{t} \omega)  z(\theta_{t} \omega)  }^2
    \right ).
\ee
By \eqref{pest3_11}-\eqref{pest3_13} we find that
there exists $K_2 =K_2(\nu) \ge K_1$ such that
for all $k \ge K_2$, 
\be
 \label{pest3_15} 
\frac{d}{dt}\int \rho(\frac{\abs{x}^2}{k^2})\abs{v}^2dx
 + 
   (
 \frac{5}{4}\lambda -2\alpha \eta  (\theta_t\omega)
   )
 \int \rho(\frac{\abs{x}^2}{k^2})
 \abs{v}^2dx 
 $$
 $$
 \leq  
    \nu (\norm{v}^p_p
  +\norm{\nabla(v+\eps hz(\theta_{t} \omega))}^p_p)
  $$
  $$
+ c_7 \int_{\abs{x} \ge k} 
( \abs{g(t,x)}^2+ \abs{\psi_1(t,x)} + \abs{\psi_3(t,x)}^{q_1})dx
$$
$$
 + \nu
  \left (
    \abs{\eps z(\theta_t\omega)}^p
    + \abs{\eps z(\theta_t\omega)}^q
    + \abs{\alpha \eps \eta  (\theta_{t} \omega)  z(\theta_{t} \omega)  }^2
    \right ).
\ee
Multiplying \eqref{pest3_15} 
by $e^{\frac{5}{4}\lambda t-2\alpha\int^{t}
_{0}\eta(\theta_{r}\omega)dr}$, 
and integrating from $\tau -t$ to $\sigma$
with $\sigma \ge \tau -t$,
 we get 
 \be
 \label{pest3_17}
\int_{\R^n} \rho(\frac{\abs{x}^2}{k^2})
\abs{v(\sigma, \tau-t,\omega,v_{\tau-t})}^2dx
 \leq e^{\frac{5}{4} \lambda (\tau -t -\sigma) 
 -2\alpha\int_{\sigma}^{\tau-t}
 \eta(\theta_{r}\omega)dr}
 \int_{\R^n} \rho(\frac{\abs{x}^2}{k^2})
 \abs{v_{\tau-t}}^2dx
 $$
 $$
+ \nu \int_{\tau-t}^{\sigma}
e^{\frac{5}{4}\lambda(s- \sigma)
-2\alpha\int_{ \sigma }^{s}
\eta(\theta_{r}\omega)dr}
(\norm{v(s, \tau-t,\omega,v_{\tau-t})}
^p_p+\norm{\nabla(v 
+\eps hz(\theta_{s} \omega))}^p_p)ds
$$
$$
+ \nu \int_{\tau-t}^{ \sigma }
e^{\frac{5}{4}\lambda(s- \sigma  )
-2\alpha\int_{\sigma }^{s}\eta(\theta_{r}\omega)dr}
(
\abs{\eps z(\theta_s\omega)}^p +
\abs{\eps z(\theta_s\omega)}^q 
+ \abs{\alpha \eps \eta(\theta_s\omega)
z(\theta_s\omega)}^2 
)ds
$$
$$
+ c_7 \int_{\tau-t}^{ \sigma  }\int_{\abs{x} \ge k }
e^{\frac{5}{4}\lambda(s- \sigma  )
-2\alpha\int_{\sigma   }^{s}\eta(\theta_{r}\omega)
dr}(\abs{g(s,x)}^2 +\abs{\psi_1(s,x)} 
+ \abs{\psi_3(s,x)}^{q_1}
 )dxds.
\ee
 Replacing $\omega$ with
  $\theta_{-\tau}\omega$ in \eqref{pest3_17},
  after simple calculations,  we  get
  for all $k\ge K_2$  and $\sigma \in [\tau -1, \tau]$,
 \be
 \label{pest3_21}
\int_{\R^n} \rho(\frac{\abs{x}^2}{k^2})
\abs{v( \sigma , \tau-t,\theta_{-\tau}\omega,v_{\tau-t})}^2dx
 \leq e^{\frac{5}{4}\lambda (\tau -t -\sigma)
 +2\alpha\int^{\sigma -\tau}_{-t}
 \eta(\theta_{r}\omega)dr}  
 \norm{v_{\tau-t}}^2 
 $$
 $$
 + \nu \int_{\tau-t}^{ \sigma }
e^{\frac{5}{4}\lambda(s- \sigma )
-2\alpha\int_{\sigma -\tau}^{s-\tau}
\eta(\theta_{r}\omega)dr}
(\norm{v(s, \tau-t,  \theta_{-\tau} \omega,v_{\tau-t})}
^p_p+\norm{\nabla(v 
+\eps hz(\theta_{s-\tau} \omega))}^p_p)ds
$$
  $$
+ \nu \int_{-\infty}^{\sigma -\tau}e^{\frac{5}{4}\lambda 
(s +\tau -\sigma) 
 - 2\alpha\int^{s}_{\sigma -\tau}\eta(\theta_{r}\omega)dr}
 (
 \abs{\eps z(\theta_s\omega)}^p
 +
 \abs{\eps z(\theta_s\omega)}^q 
 + \abs{\alpha \eps \eta(\theta_s\omega)
 z(\theta_s\omega)}^2 
 )ds
 $$
 $$
+ c_7 \int_{-\infty}^{\sigma -\tau}\int_{\abs{x} \ge k}
e^{\frac{5}{4}\lambda  (s +\tau -\sigma) 
-2\alpha\int^{s}_{\sigma -\tau}\eta(\theta_{r}\omega)dr}
(\abs{g(s+\tau,x)}^2 
+\abs{\psi_1(s+\tau,x)}
+ \abs{\psi_3(s+\tau,x)}^{q_1}
 )dxds
 $$
 $$
  \leq 
 e^{\frac{5}{4}\lambda  
+2\alpha\int_{-1}^{0}
\abs{\eta(\theta_{r}\omega)}dr}
 e^{-\frac{5}{4}\lambda t
+2\alpha\int^{0}_{-t}\eta(\theta_{r}\omega)dr}  
 \norm{v_{\tau-t}}^2 
 $$
 $$
 + \nu e^{\frac{5}{4}\lambda  
+2\alpha\int_{-1}^{0}
\abs{\eta(\theta_{r}\omega)}dr}
  \int_{\tau-t}^{ \tau }
e^{\frac{5}{4}\lambda(s- \tau )
-2\alpha\int_{0}^{s-\tau}
\eta(\theta_{r}\omega)dr}
(\norm{v }
^p_p+\norm{\nabla(v 
+\eps hz(\theta_{s-\tau} \omega))}^p_p)ds
$$
  $$
+ \nu
 e^{\frac{5}{4}\lambda  
+2\alpha\int_{-1}^{0}
\abs{\eta(\theta_{r}\omega)}dr}
  \int_{-\infty}^{0}e^{\frac{5}{4}\lambda s
 - 2\alpha\int^{s}_{0}\eta(\theta_{r}\omega)dr}
 (
 \abs{\eps z(\theta_s\omega)}^p
 +
 \abs{\eps z(\theta_s\omega)}^q 
 + \abs{\alpha \eps \eta(\theta_s\omega)
 z(\theta_s\omega)}^2 
 )ds
 $$
 $$
+ c_8 \int_{-\infty}^{0}\int_{\abs{x} \ge k}
e^{\frac{5}{4}\lambda s
-2\alpha\int^{s}_{0}\eta(\theta_{r}\omega)dr}
(\abs{g(s+\tau,x)}^2 
+\abs{\psi_1(s+\tau,x)}
+ \abs{\psi_3(s+\tau,x)}^{q_1}
 )dxds.
\ee
 Since $v_{\tau-t} \in D(\tau-t,\theta_{-\tau}\omega)$, 
 we see that for every $\nu > 0$, 
 $\tau \in \R$,
 $\omega \in \Omega$
 and $\alpha>0$,
 there exists a $T_1(\tau,\omega,D, \alpha, \nu) > 0$ 
 such that for every $t \ge T_1$ and $\sigma \in [\tau-1, \tau]$,
 \be\label{pest3_23}
  e^{\frac{5}{4}\lambda  
+2\alpha\int_{-1}^{0}
\abs{\eta(\theta_{r}\omega)}dr}
 e^{-\frac{5}{4}\lambda t
+2\alpha\int^{0}_{-t}\eta(\theta_{r}\omega)dr}  
 \norm{v_{\tau-t}}^2 
 $$
 $$ \le
 e^{\frac{5}{4}\lambda  
+2\alpha\int_{-1}^{0}
\abs{\eta(\theta_{r}\omega)}dr}
 e^{-\frac{5}{4}\lambda t
+2\alpha\int^{0}_{-t}\eta(\theta_{r}\omega)dr}  
 \norm{D({\tau-t}, \theta_{-t} \omega)}^2 
  \leq \nu .
\ee
Since  $\int_{-\infty}^{0}\ii
e^{\frac{5}{4}\lambda s
-2\alpha\int^{s}_{0}\eta(\theta_{r}\omega)dr}
(\abs{g(s+\tau,x)}^2 
+\abs{\psi_1(s+\tau,x)}
+ \abs{\psi_3(s+\tau,x)}^{q_1}
 )dxds$
 is convergent, we have
$$
 \int_{-\infty}^{0}\int_{\abs{x} \ge k}
e^{\frac{5}{4}\lambda s
-2\alpha\int^{s}_{0}\eta(\theta_{r}\omega)dr}
(\abs{g(s+\tau,x)}^2 
+\abs{\psi_1(s+\tau,x)}
+ \abs{\psi_3(s+\tau,x)}^{q_1}
 )dxds \to 0,
$$
 as $k \to \infty$. 
 Therefore,  there exists
 $K=K_3(\tau, \omega, \alpha, \nu) \ge K_2$
 such that  for all $k \ge K_3$,
 \be\label{pest3_25}
c_8  \int_{-\infty}^{0}\int_{\abs{x} \ge k}
e^{\frac{5}{4}\lambda s
-2\alpha\int^{s}_{0}\eta(\theta_{r}\omega)dr}
(\abs{g(s+\tau,x)}^2 
+\abs{\psi_1(s+\tau,x)}
+ \abs{\psi_3(s+\tau,x)}^{q_1}
 )dxds  \le \nu.
 \ee
 Note that
 $$
  \norm{v(s, \tau-t,  \theta_{-\tau} \omega,v_{\tau-t})}
^p_p
\le
2^{p}
\left (\norm{v(s, \tau-t,  \theta_{-\tau} \omega,v_{\tau-t})
+ \eps h z(\theta_{s-\tau}  \omega) }
^p_p
+ 
 \norm{ 
  \eps h z(\theta_{s-\tau} \omega) }
^p_p
\right ),
$$
which along with
\eqref{p_inequality} and
 Lemma \ref{est2} shows that
 there exists $T_2
 = T_2(\tau, \omega, D, \alpha,  \nu) \ge T_1$
 such that  for all $t \ge T_2$,
 $$
   \int_{\tau-t}^{\tau}
e^{\frac{5}{4}\lambda(s-\tau)
-2\alpha\int_{0}^{s-\tau}
\eta(\theta_{r}\omega)dr}
(\norm{v(s, \tau-t,  \theta_{-\tau} \omega,v_{\tau-t})}
^p_p+\norm{\nabla(v 
+\eps hz(\theta_{s-\tau} \omega))}^p_p)ds
$$
\be\label{pest3_30}
\le  
c_9 R(\tau, \omega, \alpha, \eps)
+    c_{10} 
 \int_{-\infty}^{0}e^{\frac{5}{4}\lambda s
 - 2\alpha\int^{s}_{0}\eta(\theta_{r}\omega)dr}
 \abs{\eps z(\theta_s\omega)}^p ds, 
\ee
where 
$R(\tau, \omega, \alpha, \eps)$
is  the number given by \eqref{est2_2}.
It follows  from \eqref{pest3_21}-\eqref{pest3_30}
that for all $k\ge K_3$,   $t \ge T_2$
and $\sigma \in [\tau -1, \tau]$,
\be
 \label{pest3_32}
\int_{\R^n} \rho(\frac{\abs{x}^2}{k^2})
\abs{v(\sigma, \tau-t,\theta_{-\tau}\omega,v_{\tau-t})}^2dx
 \le 2\nu
 +  \nu
c_{11} R(\tau, \omega, \alpha, \eps)
$$
$$
+ \nu c_{12}
    \int_{-\infty}^{0}e^{\frac{5}{4}\lambda s
 - 2\alpha\int^{s}_{0}\eta(\theta_{r}\omega)dr}
 (
 \abs{\eps z(\theta_s\omega)}^p
 +
 \abs{\eps z(\theta_s\omega)}^q 
 + \abs{\alpha \eps \eta(\theta_s\omega)
 z(\theta_s\omega)}^2 
 )ds.
\ee
Note that  $\rho(\frac{\abs{x}^2}{k^2})= 1$
when $\abs{x}^2 \ge 2 k^2$. This along with
\eqref{pest3_32}  concludes    the proof.
   \end{proof}

The asymptotic compactness of solutions of equation  \eqref{veq1}
is  given  below.

\begin{lem}
\label{asyv}
Suppose  \eqref{f1}-\eqref{f3}  and \eqref{g1} hold.
Then for every $\alpha \le \alpha_0$,
$\eps >0$,
 $\tau \in \R$, $\omega \in \Omega$   and $D=\{D(\tau, \omega)
: \tau \in \R,  \omega \in \Omega\}  \in \cald$,
 the sequence
 $v(\tau, \tau -t_n,  \theta_{-\tau} \omega,   v_{0,n}  ) $   has a
   convergent
subsequence in $\ltwo $ provided
  $t_n \to \infty$ and 
 $ v_{0,n}  \in D(\tau -t_n, \theta_{ -t_n} \omega )$.
 \end{lem}
 
 \begin{proof}
 By Lemma \ref{est1} we find  that 
 there  exists  $N_1 =N_1(\tau, \omega, D, \alpha)>0$
  such that   for all $n \ge N_1$,
   \be\label{pav_1}
   \| v(\tau -1,  \tau -t_n, \theta_{-\tau} \omega, v_{0,n} ) \|
   \le c_1.
  \ee
  Applying Lemma \ref{comv1} to the sequence
  $v(\tau,  \tau - 1, \theta_{-\tau} \omega,
    v(\tau -1,  \tau -t_n, \theta_{-\tau} \omega, v_{0,n} )  )$,
    we find  that there exist $s_0 \in (\tau -1, \tau)$, 
    $v_0 \in \ltwo$
    and a subsequence (not relabeled) such that
    as $n \to \infty$,
   $$
    v(s_0,  \tau - 1, \theta_{-\tau} \omega,
    v(\tau -1,  \tau -t_n, \theta_{-\tau} \omega, v_{0,n} )  )
    \to v_0
    \quad \text{ in } L^2(\o_k)
    \text{ for every }  k \in \N,
   $$
   that is,   as $n \to \infty$,
    \be
    \label{pav_2}
    v(s_0,    \tau -t_n, \theta_{-\tau} \omega, v_{0,n} )
    \to v_0
    \quad \text{ in } L^2(\o_k)
    \text{ for every }  k \in \N.
    \ee
    By \eqref{pexiv_30}
   we get
    $$
    \norm{ v(\tau,  s_0, \theta_{-\tau} \omega,
    v(s_0,  \tau -t_n, \theta_{-\tau} \omega, v_{0,n} ) ) 
    - v(\tau, s_0,  \theta_{-\tau} \omega, v_0)} 
    $$
    $$
    \le
    e^{c_1 (\tau -s_0)}
    \norm{  
   v(s_0,  \tau -t_n, \theta_{-\tau} \omega, v_{0,n} )  -v_0  
    }.
    $$ 
  Since $s_0 \in (\tau-1, \tau)$, we obtain
     \be\label{pav_2a}
    \norm{ v(\tau,  s_0, \theta_{-\tau} \omega,
    v(s_0,  \tau -t_n, \theta_{-\tau} \omega, v_{0,n} ) ) 
    - v(\tau, s_0,  \theta_{-\tau} \omega, v_0)} ^2
    $$
    $$
    \le
     e^{2c_1  }
    \int_{\abs{x} < k}
     \abs{v(s_0,  \tau -t_n, \theta_{-\tau} \omega, v_{0,n} )
     -v_0 }^2 dx
     $$
     $$
     +
      e^{2c_1  }
    \int_{\abs{x} \ge k}
     \abs{v(s_0,  \tau -t_n, \theta_{-\tau} \omega, v_{0,n} )
     -v_0 }^2 dx
     $$
      $$
    \le
     e^{2c_1  }
    \int_{\abs{x} < k}
     \abs{v(s_0,  \tau -t_n, \theta_{-\tau} \omega, v_{0,n} )
     -v_0 }^2 dx
     $$
     $$
     +
     2 e^{2c_1  }
    \int_{\abs{x} \ge k}
    \left (
     \abs{v(s_0,  \tau -t_n, \theta_{-\tau} \omega, v_{0,n} )}^2
     + \abs{v_0}^2 
     \right )
       dx.
     \ee
     Since  $v_0 \in \ltwo$, given $\nu>0$,  there exists
     $K_1 =K_1(\nu) \ge 1$ such that for all $k \ge K_1$,
     \be\label{pav_3}
      2 e^{2c_1  }
    \int_{\abs{x} \ge k}
      \abs{v_0}^2  ds \le \nu.
      \ee
   On the other hand,  by Lemma \ref{est3},
   there exist $N_2=N_2(\tau, \omega, D,  \alpha, \eps, \nu)\ge 1$ 
     and
     $K_2=K_2(\tau, \omega,   \alpha, \eps, \nu)\ge  K_1$ 
     such that  for all $n \ge N_2$  and $k \ge K_2$,
     \be\label{pav_4}
       2  e^{2c_1  }
    \int_{\abs{x} \ge  k}
     \abs{v(s_0,  \tau -t_n, \theta_{-\tau} \omega, v_{0,n} )
       }^2 dx
     \le \nu.
     \ee
     By \eqref{pav_2} we find that there exists
     $N_3 = N_3(\tau, \omega, D,  \alpha, \eps, \nu)\ge N_2$ 
     such that
     for all $n \ge N_3$,
     \be\label{pav_8}
      e^{2c_1  }
    \int_{\abs{x} < K_2}
     \abs{v(s_0,  \tau -t_n, \theta_{-\tau} \omega, v_{0,n} )
     -v_0 }^2 dx
     \le \nu.
     \ee
     It follows   from \eqref{pav_2a}-\eqref{pav_8}
     that  for all $n \ge N_3$, 
       $$
    \norm{ v(\tau,  s_0, \theta_{-\tau} \omega,
    v(s_0,  \tau -t_n, \theta_{-\tau} \omega, v_{0,n} ) ) 
    - v(\tau, s_0,  \theta_{-\tau} \omega, v_0)} ^2
    \le 3\nu,
    $$
    that is, for all $n \ge N_3$, 
      $$
    \norm{ v(\tau,   \tau -t_n, \theta_{-\tau} \omega, v_{0,n}  ) 
    - v(\tau, s_0,  \theta_{-\tau} \omega, v_0)} ^2
    \le 3\nu.
    $$
    Therefore, 
    $ v(\tau,   \tau -t_n, \theta_{-\tau} \omega, v_{0,n}  ) $
    converges to  $  v(\tau, s_0,  \theta_{-\tau} \omega, v_0) $
    in $\ltwo$. This completes   the proof.
   \end{proof}

\section{ Random Attractors}
\setcounter{equation}{0}

  In this section, we prove the existence of $\cald$-pullback attractor
  for \eqref{seq1}-\eqref{seq2} in $\ltwo$
  by  Proposition \ref{att}.
  To this end, we need to establish   the
  existence of    $\cald$-pullback absorbing sets
  and the $\cald$-pullback asymptotic compactness
  of $\Phi$ in $\ltwo$. 
  The existence of  absorbing sets of $\Phi$ is given below.

  \begin{lem}
  \label{lem41}
    Suppose \eqref{f1}-\eqref{f3}
    and \eqref{g1} hold. Then    for every
    $\alpha \le \alpha_0$  and $\eps >0$,  
    the stochastic equation \eqref{seq1}
    with \eqref{seq2}  has   a closed
    measurable $\cald$-pullback absorbing set $K=\{ K(\tau, \omega):
    \tau \in \R, \omega \in \Omega \} \in \cald$ which  is given by
    \be\label{lem41_1}
    K (\tau, \omega) = \{ u \in \ltwo: \| u\|^2 \le
    2\norm{\eps h z(\omega)}^2
    +
    2 R(\tau, \omega, \alpha, \eps)
    \},
    \ee
    where $R(\tau, \omega, \alpha, \eps)$ is the number given by
    \eqref{est2_2}.
  \end{lem}
  
  \begin{proof}
  Let $D =\{ D(\tau, \omega): \tau \in \R, \omega \in
  \Omega\} \in \cald$. For every $\tau \in \R$   and
  $\omega \in \Omega$, denote by
 \be\label{plem41_1a}
  {\widetilde{D}}(\tau, \omega)
  =\{ v \in \ltwo: v= u -\eps h z(\omega)
  \text{ for some } u \in D(\tau, \omega) \}.
  \ee
  Since $z$ is tempered, we find that
  the family 
  ${\widetilde{D}}
  = \{{\widetilde{D}}(\tau, \omega), \tau \in \R, \omega \in
  \Omega  \}$ belongs to $\cald$ provided
  $D\in \cald$.
  By \eqref{uv} we have
  \be\label{plem41_1}
  u(\tau, \tau -t, \theta_{-\tau} \omega, u_{\tau -t})
  = 
   v(\tau, \tau -t, \theta_{-\tau} \omega, v_{\tau -t})
   + \eps h z(\omega)  \ 
   \text{ with } \  v_{\tau -t}
   = u_{\tau -t} -\eps h z(\theta_{-t} \omega).
   \ee
   Thus, if $u_{\tau -t}
   \in D(\tau-t, \theta_{-t} \omega) \in \cald$, then
   $v_{\tau -t}
   \in {\widetilde{D}}(\tau-t, \theta_{-t} \omega)
   \in \cald$. 
   By Lemma \ref{est2} we find that
   there exists  $T=T(\tau, \omega, D, \alpha, \eps) >0$
   such that for all $t \ge T$,
    $$\norm{v(\tau, \tau -t, \theta_{-\tau} \omega, v_{\tau -t})}
    ^2
    \le  R(\tau, \omega, \alpha, \eps),
    $$
    where $R(\tau, \omega, \alpha, \eps)$
    is as in \eqref{est2_2}. 
    By \eqref{plem41_1} we get for all $t \ge T$,
    $$
    \norm{u(\tau, \tau -t, \theta_{-\tau} \omega, u_{\tau -t})}
    ^2
    \le  2 \norm{\eps h z(\omega) }^2
    + 2R(\tau, \omega, \alpha, \eps ).
    $$
    This along with 
    \eqref{pcycle2} and \eqref{lem41_1} shows
     that  for all $t \ge T$,
    \be\label{plem41_4}
    \Phi (t, \tau -t, \theta_{-t} \omega,
    D(\tau-t, \theta_{-t} \omega) )
    \subseteq K(\tau, \omega).
    \ee
    On the other hand, by \eqref{est2_3} and the temperedness
    of $z$ we obtain
    \be\label{plem41_6}
    \lim_{t \to \infty}
  e^{-{\frac 54} \lambda t + 2 \alpha \int_{-t}^0 \eta (\theta_r\omega) dr}  
  \norm{
    K(\tau- t, \theta_{-t} \omega)} =0.
    \ee
    By \eqref{plem41_4}-\eqref{plem41_6} we find that
    $K $ given by \eqref{lem41_1} is a closed
    $\cald$-pullback absorbing set of $\Phi$ in $\cald$.
    Note that   the measurability of  $K (\tau, \omega)$ 
    in $\omega \in \Omega$  follows  from that of 
      $z(\omega)$  and 
  $R(\tau, \omega, \alpha, \eps)$  immediately.
  This  completes     the proof.
     \end{proof}

    The following is our main result regarding the 
     existence of
    $\cald$-pullback attractors of $\Phi$.
    
    \begin{thm}
    \label{eatt}
    Suppose  \eqref{f1}-\eqref{f3}  and \eqref{g1} hold.
    Then for every $\alpha \le \alpha_0$  and $\eps>0$,
 the   stochastic equation \eqref{seq1}
 with  \eqref{seq2}  
   has a unique $\cald$-pullback attractor $\cala
   =\{\cala(\tau, \omega):
      \tau \in \R, \ \omega \in \Omega \} \in \cald$
 in $\ltwo$.   
 In addition,  if there is $T>0$  such that
 $f(t,x,s)$, $g(t,x)$,
 $\psi_1(t,x)$ and $\psi_3 (t,x)$ are all
 $T$-periodic in $t$
 for fixed $x \in \R^n$
 and $s\in \R$,   then the attractor $\cala$
 is also $T$-periodic.
 \end{thm}

\begin{proof}
We first prove  that  $\Phi$ is    $\cald$-pullback
asymptotically  compact    in $\ltwo$;
that is, 
  for every  
 $\tau \in \R$, $\omega \in \Omega$,
 $D  \in \cald$,  $t_n \to \infty$
 and $ u_{0,n}  \in D(\tau -t_n, \theta_{ -t_n} \omega )$,
we want  to show that 
 the sequence
 $\Phi (t_n, \tau -t_n,  \theta_{-t_n} \omega,   u_{0,n}  )$
  has a
   convergent
subsequence in $\ltwo $.
 Let  $ v_{0,n} 
 = u_{0,n} - \eps hz(\theta_{-t_n} \omega)$
 and  $\widetilde{D}$ be the family
  given by \eqref{plem41_1a}.
 Since   $ u_{0,n}  \in D(\tau -t_n, \theta_{ -t_n} \omega )$, we find
 that $ v_{0,n} 
   \in {\widetilde{D}}(\tau -t_n, \theta_{ -t_n} \omega )
  \in \cald$. Therefore,  by 
     \eqref{plem41_1}  and 
    Lemma \ref{asyv} we find that
  $u (\tau, \tau -t_n,  \theta_{-\tau} \omega,   u_{0,n}  )$
  has a convergent subsequence in $\ltwo$.
   This together with \eqref{pcycle2} indicates that
    $\Phi (t_n, \tau -t_n,  \theta_{-t_n} \omega,   u_{0,n}  )$
  has a
   convergent
subsequence,  and thus it is
$\cald$-pullback asymptotically compact  in $\ltwo $.
Since $\Phi$ also has a
closed measurable  $\cald$-pullback
absorbing  set $K$ given by \eqref{lem41_1},
by Proposition \ref{att} we get the existence and uniqueness
of $\cald$-pullback attractor $\cala \in \cald$ of $\Phi$
immediately.

Next, we discuss   $T$-periodicity of $\cala$.
Note  that  if  $f$  and $g$ are $T$-periodic in
their first arguments, then the cocycle $\Phi$
is also  $T$-periodic. 
Indeed,  in this case,    
for every
 $t \in \R^+$, $\tau \in \R$ and $\omega \in \Omega$, 
 by \eqref{pcycle2} we have
\be\label{peatt_1}
\Phi (t, \tau +T, \omega,  \cdot )
= u(t+ \tau +T, \tau +T,  \theta_{ -\tau -T} \omega,  \cdot )
=u(t +\tau, \tau, \theta_{ -\tau} \omega,  \cdot )
= \Phi (t, \tau,  \omega,  \cdot ).
\ee
 In addition,  if  $g(t,x)$, $\psi_1(t,x)$ and $\psi_3(t,x)$ are
 all $T$-periodic in $t$, then by
 \eqref{est2_2} and \eqref{lem41_1} we  get
  $ 
K(\tau +T, \omega)
=K(\tau, \omega )$
 for all $\tau \in \R$ and
$\omega \in \Omega$.
This along with
\eqref{peatt_1}
and Proposition \ref{att}
yields  the
$T$-periodicity of $\cala$.
 \end{proof}


\begin{thebibliography}{99}


\bibitem{adi1}
A. Adili and B.  Wang, 
Random attractors for stochastic FitzHugh-Nagumo systems driven by deterministic non-autonomous forcing, 
{\em Continuous and Discrete Dynamical Systems Series B}, 
{\bf 18} (2013),  643-666.   


\bibitem{arn1}
L. Arnold, {\em Random Dynamical Systems}, Springer-Verlag, 1998.

 
%
%
%
%
%


\bibitem{bat1}
P.W. Bates,  H. Lisei and  K.  Lu,
 Attractors for stochastic lattice
 dynamical systems,
{\em Stoch. Dyn.},   {\bf 6}  (2006),      1-21.




\bibitem{bat2}
P.W. Bates,   K.  Lu   and B. Wang,
 Random attractors for  stochastic reaction-diffusion equations
on unbounded domains,  {\em J. Differential Equations},
 {\bf  246}   (2009),   845-869.



\bibitem{bat3}
P.W. Bates,   K.  Lu   and B. Wang,
Tempered random attractors for  parabolic equations
in weighted spaces, 
  {\em J.   Math.  Phys.},
 {\bf  54}   (2013),   081505, 1-26.



\bibitem{beyn1}
W.J. Beyn,  B. Gess,  P. Lescot and
M. R$\ddot{o}$ckner,
The global random attractor for  a class
of stochastic porous media equations,
{\em Comm. Partial Differential Equations},
{\bf 36} (2011), 446-469.



  \bibitem{car1}
    T. Caraballo, M.J. Garrido-Atienza, B. Schmalfuss
    and J. Valero,
    Non-autonomous and random attractors for
    delay random semilinear equations without
    uniqueness,
   {\em Discrete Continuous Dynamical Systems},
  {\bf 21}  (2008),   415-443.
  
  
   
\bibitem{car2}
T. Caraballo, J. Real, I.D. Chueshov,
 Pullback attractors for stochastic heat
 equations in materials with memory,
  {\em Discrete Continuous Dynamical Systems B},
  {\bf 9}  (2008),   525-539.

 
 \bibitem{car3}
 T. Caraballo  and  J. A. Langa,
 On the upper semicontinuity  of cocycle attractors
 for non-autonomous  and random
 dynamical systems,
 {\em Dynamics of Continuous, Discrete and Impulsive
 Systems  Series A: Mathematical Analysis},
 {\bf   10} (2003), 491-513.

 
  
   \bibitem{car4}
 T.   Caraballo, M. J.  Garrido-Atienza,
 B.  Schmalfuss  and
 J. Valero,  Asymptotic behaviour of a stochastic 
 semilinear dissipative functional equation 
 without uniqueness of solutions,
 {\em  Discrete Contin. Dyn. Syst. Ser. B},
 {\bf  14}  (2010),  439-455.
 
 
 \bibitem{car5}
  T. Caraballo, M. J.  Garrido-Atienza
  and T.  Taniguchi,
  The existence and exponential behavior of 
  solutions to stochastic delay evolution 
  equations with a fractional Brownian motion,
  {\em  Nonlinear Anal.},
  {\bf  74}  (2011),  3671-3684.
  
  \bibitem{car6}
  T. Caraballo,  J. A.   Langa,  V.S. Melnik
  and J. Valero,
  Pullback  attractors for nonautonomous  and stochastic 
  multivalued dynamical systems,
  {\em Set-Valued Analysis},  {\bf 11} (2003),
  153-201.

 




\bibitem{chu2}
I. Chueshov and M. Scheutzow,
On the structure of attractors  and invariant measures for a class of
monotone  random systems,
{\em Dynamical Systems}, {\bf 19} (2004), 127-144.


  \bibitem{chu3}
    I. Chueshow, {\em Monotone Random Systems - Theory
    and Applications}, Lecture Notes in Mathematics 1779,
    Springer, Berlin, 2001.
    
 
 

    \bibitem{cra1}
    H. Crauel,  A. Debussche and
    F. Flandoli, Random attractors,
    {\em J. Dyn. Diff. Eqns.}, {\bf 9} (1997), 307-341.


    \bibitem{cra2}
    H. Crauel  and
    F. Flandoli,  Attractors for random dynamical systems,
    {\em Probab. Th. Re. Fields}, {\bf 100} (1994), 365-393.

%
%
%

 
    \bibitem{dua3}
J. Duan and B. Schmalfuss,
 The 3D quasigeostrophic fluid dynamics
  under random forcing on boundary,
  {\em Comm. Math. Sci.},
  {\bf 1} (2003),  133-151.
    

  


     \bibitem{fla1}
    F. Flandoli and B. Schmalfuss,
Random attractors for
the 3D stochastic Navier-Stokes equation with multiplicative
noise,
    {\em Stoch. Stoch. Rep.}, {\bf 59} (1996),  21-45.



\bibitem{gar1}
M.J. Garrido-Atienza and
B.  Schmalfuss,
Ergodicity of the infinite dimensional fractional Brownian motion,
{\em  J. Dynam. Differential Equations},
{\em  23}  (2011), 671-681.


\bibitem{gar2}
M.J. Garrido-Atienza,
A.  Ogrowsky  and B.  Schmalfuss,
 Random differential equations with 
 random delays,
 {\em  Stoch. Dyn.},
 {\bf  11}  (2011), 369-388.


\bibitem{gar3}
M.J. Garrido-Atienza,
B. Maslowski  and B.  Schmalfuss,
Random attractors for stochastic  equations
driven by a fractional  Brownian motion,
{\em International J. Bifurcation and Chaos},
{\bf 20} (2010),  2761-2782.



\bibitem{gess1}
B. Gess,  W. Liu and
M. Rockner,
Random attractors  for  a class
of stochastic partial differential equations
driven by general additive noise,
{\em J.  Differential Equations},
{\bf  251} (2011),  1225-1253.



\bibitem{gess2}
B. Gess, 
Random attractors  for 
degenerate stochastic partial differential equations,
{\em J.  Dynamics and   Differential Equations},
{\bf  25} (2013),  121-157.


\bibitem{gess3}
B. Gess,
Random attractors  for singular stochastic evolution equations,
{\em J. Differential Equations}, {\bf 255} (2013),
524-559.

 
 
  

\bibitem{huang1}
J. Huang and W. Shen,
Pullback attractors for nonautonomous and random 
parabolic equations on non-smooth domains,
{\em Discrete and Continuous Dynamical Systems},
{\bf 24} (2009),   855-882.


  
  

\bibitem{kloe1}
P.E. Kloeden and J.A. Langa,
 Flattening, squeezing and the existence of
random attractors,
{\em Proc. Royal Soc. London Serie A.}, {\bf 463}  (2007), 163-181.

 \bibitem{kloe2}
P.E. Kloeden  and M. Rasmussen,
{\em Nonautonomous  Dynamical Systems},
Mathematical Surveys and Monographs, Vol. 176,
Amer. Math. Soc.,
Providence, 2011.

 
 
 
 \bibitem{lion1}
 J.L. Lions, 
 {\em Quelques Methodes de Resolution des
 Problemes aux Limites Non Lineaires},
 Dunod, Paris, 1969.
 

 \bibitem{lv1} Y. Lv  and  W. Wang, 
 Limiting dynamics for stochastic wave equations,
 {\em J. Differential Equations}, 
 {\bf 244} (2008), 1-23.
 
 
  

  
\bibitem{schm1}
B.  Schmalfuss, Backward cocycles  and attractors  
of stochastic differential equations,
{\em International Seminar on Applied Mathematics-Nonlinear 
Dynamics: Attractor Approximation and Global Behavior},    185-192,
Dresden,
1992.
 

\bibitem{shen1}
Z. Shen, S. Zhou and W. Shen,
One-dimensional random attractor and rotation
number of the stochastic damped sine-Gordon
equation,
{\em J. Differential Equations},
{\bf 248} (2010),  1432-1457.

\bibitem{show1}
R.E. Showalter,
{\em  Monotone Operators in Banach Space
and Nonlinear Partial Differential Equations},
American Mathematical Society,  Providence, 1997.
 




 \bibitem{wan1}
 B. Wang,
 Random Attractors for the Stochastic Benjamin-Bona-Mahony Equation on Unbounded Domains,
{ \em J. Differential Equations},  {\bf 246} (2009), 2506-2537.


 \bibitem{wan2}
 B.  Wang,
Asymptotic behavior of stochastic wave equations with critical
exponents on $\R^3$,  
{\em  Transactions of  American Mathematical Society}, 
{\bf 363} (2011), 3639-3663.
 

   

\bibitem{wan5}
      B. Wang,
      Sufficient and necessary criteria for
      existence of pullback attractors for
      non-compact random dynamical systems,
     {\em  J. Differential Equations},
      {\bf 253}  (2012),1544-1583.
  
%


      
   \bibitem{wan7}
   B. Wang, 
 Random attractors for non-autonomous stochastic
wave equations with multiplicative noise,
{\em Discrete and Continuous Dynamical Systems  Series A},
{\bf  34} (2014),  269-300.
 
 
 
 \bibitem{wan8}
 B. Wang    and  B. Guo,
 Asymptotic behavior of non-autonomous stochastic parabolic
equations with  nonlinear Laplacian  principal part,
{\em Electronic J. Differential Equations},  {\bf 2013}  (2013),
No.  191,   1-25.  
 
 
%
%
%
%
%
%
%
%
%


%
%


 \end{thebibliography}
\end{document}